\documentclass[reqno,11pt]{amsart}
\usepackage{amsmath,amssymb,amsthm, latexsym,cancel,rotating}

\usepackage{tikz-cd}

\usetikzlibrary{cd}

\usepackage{setspace}
\usepackage{dcpic,pictexwd}
\usepackage[all]{xy}
\usepackage[pdftex]{hyperref}
\usepackage{bbm}

\numberwithin{equation}{section}

\newtheorem{theorem}{Theorem}[section]
\newtheorem{proposition}[theorem]{Proposition}

\newtheorem{corollary}[theorem]{Corollary}
\newtheorem{lemma}[theorem]{Lemma}

\newtheorem{remark}[theorem]{Remark}

\newtheorem{definition}[theorem]{Definition}

\input xy
\xyoption{poly}
\xyoption{2cell}
\xyoption{all}

\def\Ext{\mbox{\rm Ext}\,} \def\Hom{\mbox{\rm Hom}} \def\dim{\mbox{\rm dim}\,}

\def\End{\mbox{\rm End}\,}

\def\rad{\mbox{\rm rad}\,}\def\soc{\mbox{\rm soc}\,}

\def\Dim{\mbox{\rm \textbf{dim}}\,}

\def\ZZ{\mathbb Z}

\def\End{\operatorname{End}}
\def\ker{\operatorname{ker}}

\def\rad{\operatorname{rad}}

\def\soc{\operatorname{soc}}

\def\ZZ{\mathbb Z}

\def\ZZ{\mathbb{Z}}

\def\dim{\operatorname{dim}}

\def\Hom{\operatorname{Hom}}
\def\Ext{\operatorname{Ext}}

\def\half{\frac{1}{2}}

\renewcommand{\eqref}[1]{{\rm (\ref{#1})}}

\begin{document}

\title[Atomic basis of quantum cluster algebra of  type $\widetilde{A}_{2n-1,1}$]
{Atomic basis of quantum cluster algebra of  type $\widetilde{A}_{2n-1,1}$}

\author{Ming Ding, Fan Xu and Xueqing Chen}

\address{School of Mathematics and Information Science\\
Guangzhou University, Guangzhou, P. R. China}
\email{m-ding04@mails.tsinghua.edu.cn (M.Ding)}
\address{Department of Mathematical Sciences\\
Tsinghua University\\
Beijing 100084, P. R. China} \email{fanxu@mail.tsinghua.edu.cn (F.Xu)}
\address{Department of Mathematics,
 University of Wisconsin-Whitewater\\
800 W. Main Street, Whitewater, WI.53190. USA}
\email{chenx@uww.edu (X.Chen)}


\thanks{Ming Ding was supported by NSF of China (No. 11771217), Fan Xu was supported by NSF of China (No. 11471177) and Xueqing Chen was supported by Strategic Priorities Endowment Fund from Department of Mathematics UW-Whitewater.}



\keywords{quantum cluster algebra, cluster multiplication formula, positivity}


\begin{abstract}
Let $Q$ be the affine quiver of type $\widetilde{A}_{2n-1,1}$  and $\mathcal{A}_{q}(Q)$ be  the  quantum cluster algebra associated to the valued quiver $(Q,(2,2,\dots,2))$.  We prove some cluster multiplication formulas, and deduce  that the cluster variables associated with vertices of $Q$ satisfy a quantum analogue of the constant coefficient linear relations. We then construct two bar-invariant  $\mathbb{Z}[q^{\pm\frac{1}{2}}]$-bases $\mathcal{B}$ and $\mathcal{S}$ of $\mathcal{A}_{q}(Q)$ consisting of positive elements, and prove that $\mathcal{B}$ is an atomic basis. 
\end{abstract}

\maketitle


\section{Introduction}

Cluster algebras were invented and investigated by Fomin and Zelevinsky~\cite{ca1,ca2} in order
to develop an algebraic framework for understanding total positivity and canonical bases in quantum groups.
As a noncommutative analogue of cluster algebras, the concept of quantum cluster algebras was introduced by Berenstein and Zelevinsky in ~\cite{BZ2005}. The theory of cluster algebras and quantum cluster algebras has a close link to many other areas such as representation theory, Poisson geometry, Lie theory and combinatorics. The link between cluster algebras and quiver representations via cluster categories \cite{BMRRT} is explicitly characterized by the Caldero-Chapoton map \cite{CC} and the Caldero-Keller multiplication theorems \cite{CK-2}. The Caldero-Chapoton map associates the objects in the cluster
categories to some Laurent polynomials, in particular, sends indecomposable rigid objects to cluster variables.  Rupel \cite{rupel} defined a quantum version of the
Caldero-Chapoton map for the quantum cluster algebras over finite fields associated with valued acyclic quivers. For acyclic equally valued quivers,  Qin ~\cite{fanqin}
proved that cluster variables are images of indecomposable rigid objects under the quantum Caldero-Chapoton formula. In \cite{rupel-2}, Rupel extended these results to the quantum cluster algebras over finite fields associated with all valued acyclic quivers.

Let $Q$ be a finite connected quiver with $n$ vertices and without oriented cycles of length $1$ and $2$ and ${\bf{x}} = (x_1, \ldots, x_n)$ be a $n$-tuple of variables. The pair $(Q,{\bf{x}})$ is called the {cluster} of the seed. Through {mutation},  one can define recursively a family of seeds. The (coefficient--free) {cluster algebra }$\mathcal A(Q)$ is the $\mathbb{Z}$-subalgebra of the {ambient field} $\mathbb{Q}(x_1, \ldots, x_n)$ generated by all the cluster variables of the seeds arising from mutation. The so--called Laurent phenomenon tells that $\mathcal A(Q)$ is a subring of $\mathbb{Z}[c_1^{\pm 1}, \ldots, c_n^{\pm 1}]$ 
 for any cluster ${\bf{c}}= (c_1, \ldots, c_n)$ in $\mathcal A(Q)$.  An element in $\mathcal A(Q)$ is called {positive} if it belongs to the semiring $\mathbb{Z}_{\geq 0}[c_1^{\pm 1}, \ldots, c_n^{\pm 1}]$ for any cluster ${\bf{c}} = (c_1, \ldots, c_n)$ in $\mathcal A(Q)$.  Denoted by $\mathcal A(Q)^+$ the cone of positive elements in $\mathcal A(Q)$. An {atomic basis} (or a {canonically positive basis}) of $\mathcal A(Q)$ is a $\mathbb{Z}$-basis $\mathcal{B}$ of $\mathcal A(Q)$ such that $\mathcal A(Q)^+ = \bigoplus_{b \in \mathcal{B}} \mathbb{Z}_{\geq 0}b.$ 
It follows immediately from the definition that the existence of an atomic basis  implies that it consists of the positive indecomposable
elements in the cluster algebra, i.e., those non-zero positive elements that cannot be written
as a non-trivial sum of positive elements. The problem of showing the existence of the atomic basis of $\mathcal A(Q)$ remains wide open in general. For rank 2 cluster algebras of finite and affine types, Sherman and Zelevinsky introduced and constructed atomic bases which were originally called canonical bases \cite{SZ}. Cerulli constructed an atomic basis for the cluster algebra of type $A_{2}^{(1)}$ in~\cite{CI-1}. He also proved that the atomic basis coincides with the set of cluster monomials of $\mathcal A(Q)$ if $\mathcal A(Q)$ is of finite type~\cite{CI-2}. If $\mathcal A(Q)$ is not of finite type, the set of cluster monomials is not enough to be a good basis.  Dupont and Thomas~\cite{DTh} constructed the atomic bases of the cluster algebras of arbitrary quivers of affine type $\widetilde A$. Meanwhile, they also provided a new, short and elementary proof of Cerulli's result for cluster algebras of type $A$.

For the quantum analogue, the atomic basis of the quantum cluster algebra  of the Kronecker quiver,  i.e., type $\widetilde{A}_{1,1}$ was constructed in \cite{dx}
and of type $A_{2}^{(2)}$ was constructed in \cite{BCDX}, respectively.  It is natural to ask whether there exist the atomic bases for the quantum cluster algebras of finite and affine types.

In this paper, we focus on the quantum cluster algebra of the quiver $Q$ of type $\widetilde{A}_{2n-1,1}$, $n\geq 1$.  As the Chebyshev polynomials of the first kind  $F_{m}(x)$ are used to construct the atomic bases in classical cluster algebras, it is important to study them in the quantum cases. By using the quantum multiplication formulas  proved in \cite{dx-2,fanqin,rupel-2}and the positivity of quantum cluster variables~\cite{Davison,KQ}, we  construct two bar-invariant  $\mathbb{Z}[q^{\pm\frac{1}{2}}]$-bases of  the quantum cluster algebra $\mathcal{A}_{q}(Q)$ consisting of positive elements. One of these bases denoted by $\mathcal{B}$ as the atomic basis of this quantum cluster algebra is  explicitly described.
 By specializing $q$ to $1$,  $\mathcal{B}$ is exactly an atomic basis for cluster
algebra $\mathcal{A}(Q)$ proved in \cite{DTh}. The construction of the basis $\mathcal{B}$ was strongly relied on the quantum analogue of the constant coefficient linear relations. Firstly Lemmas~\ref{induc1} and ~\ref{induc11} provide a quantum analogue of the linear relations of the frieze sequences of cluster algebras associated with the vertices of the affine quiver $\widetilde{A}_{2n-1,1}$ as discussed by Keller and Scherotzke in \cite[Theorem 8.1.(b)]{KS}. Secondly Theorem~\ref{linear} can be considered as a quantum analogue of  \cite[Theorem 1.1]{P} in which Pallister showed that the cluster variables of the cluster algebras of affine types satisfy linear recurrence relations with periodic coefficients, which imply the constant coefficient relations found by Keller and Scherotzke~\cite{KS}. For the type $\widetilde{A}_{2n-1,1}$, we find an explicit representation--theoretic interpretation of the quantity $\mathcal{K}$ in  \cite[Theorem 1.1]{P} which  is nothing but $F_{2n}(X_\delta)$, i.e., the $2n$-th Chebyshev polynomial of the first kind on $X_\delta$. Most recently, Davison and Mandel~\cite{DM} construct ``quantum theta bases",  extending the set of quantum cluster monomials, for various versions of skew-symmetric quantum cluster algebras. 
These bases consist precisely of the indecomposable universally positive elements of the algebras they generate. The atomic basis of the quantum cluster algebra of type $\widetilde{A}_{2n-1,1}$ precisely consists of all indecomposable positive elements, which is a quantum version of bracelets basis of~\cite{MSW} and is expected to coincide with ``quantum theta basis".  We expect the method used in this paper can be applied on the other affine types and finite types, probably we will add some frozen vertices on the corresponding quivers.

The paper is organized as follows.  In Section 2 we recall the definitions of quantum cluster algebras and describe the quantum cluster characters for acyclic  quiver of type $\widetilde{A}_{2n-1,1}$. We provide linear relations for quantum cluster variables of type $\widetilde{A}_{2n-1,1}$ in Section 3. Then we construct the bar--invariant basis $\mathcal{B}$ of the quantum cluster algebra of type $\widetilde{A}_{2n-1,1}$ which is proved to be an atomic basis, and another bar-invariant basis $\mathcal{S}$ consisting of positive elements in Section 4. In the appendix, we provide a proof of the statement that the shift functor on the cluster category of the equal--valued acyclic quiver of full rank induces an automorphism of the corresponding quantum cluster algebra.

\section{Preliminaries}
One can refer to \cite{BZ2005,fanqin,rupel-2} for more details about the definitions of quantum cluster algebras and quantum cluster characters for acyclic valued quivers. 
\subsection{Quantum cluster algebras}
Let $\Lambda: \mathbb{Z}^{m} \times \mathbb{Z}^{m} \rightarrow \mathbb{Z}$ be a skew-symmetric bilinear form on $ \mathbb{Z}^{m}$. Denote by $\{{\bf e}_1,\cdots,{\bf e}_m\}$ the standard basis vectors in  $\mathbb{Z}^{m}$.  Let $\mathfrak{q}$ be an inderterninate and $\ZZ[\mathfrak{q}^{\pm\frac{1}{2}}]$ the ring of integer Laurent polynomials. The  quantum torus  $\mathcal{T}_{\mathfrak{q}}$ associated to $\Lambda$ is the $\ZZ[\mathfrak{q}^{\pm\frac{1}{2}}]$-algebra freely generated by the set $\{X^{\bf e}: {\bf e}\in \mathbb{Z}^{m}\}$ with multiplication defined by
$$X^{\bf e}X^{\bf f}=\mathfrak{q}^{\Lambda({\bf e},{\bf f})/2}X^{{\bf e+f}}\,\,\, \,\,\,\, (\bf{e}, \bf{f} \in  \mathbb{Z}^m).$$
Note  that $\mathcal{T}_{\mathfrak{q}}$ is an Ore domain, and  is contained in its
skew-field of fractions $\mathcal{F}_{\mathfrak{q}}$. Without causing confusion, we identify $\Lambda$ with  the $m\times m$ skew-symmetric integer matrix associated to the bilinear form $\Lambda$.

Let $\tilde{B}=(b_{ij})$ be an $m\times n$ integer matrix with $n\le m$ and $\tilde{B}^{tr}$ the transpose of  $\tilde{B}$.  The upper $n\times n$ submatrix of $\tilde{B}$ is denoted by $B$, which is called the principle part of $\tilde{B}$. The pair
$(\Lambda, \tilde{B})$ is called {\em compatible} if $\tilde{B}^{tr}\Lambda=(D|0)$ for some diagonal matrix
$D$ with positive entries. An {\em initial  quantum seed} $(\Lambda, \tilde{B}, X)$ of $\mathcal{F}_{\mathfrak{q}}$ 
  consists of a compatible pair $(\Lambda,\tilde{B})$ and the set  $X=\{X_1,\cdots,X_m\}$, where $X_i:=X^{{\bf e_i}}$ for $1\leq i\leq m$. For any $1\leq k\leq n$, one can define the mutation $\mu_k$ of the quantum seed $(\Lambda, \tilde{B}, X)$ in direction $k$ to be the new quantum seed $(\Lambda',\tilde{B}',X'):=\mu_k (\Lambda, \tilde{B}, X)$, where 

(1)\ $\Lambda'=E^{tr}\Lambda E$, where the
$m\times m$ matrix $E=(e_{ij})$ is given by
\[e_{ij}=\begin{cases}
\delta_{ij} & \text{if $j\ne k$;}\\
-1 & \text{if $i=j=k$;}\\
\max(0,-b_{ik}) & \text{if $i\ne j = k$.}
\end{cases}
\]

(2)\ $\tilde{B}'=(b'_{ij})$ is given by
\[b'_{ij}=\begin{cases}
-b_{ij} & \text{if $i=k$ or $j=k$;}\\
b_{ij}+[b_{ik}]_{+}b_{kj} +b_{ik}[-b_{kj}]_{+}& \text{otherwise.}
\end{cases}
\]

(3)\  $X'=\{X'_1,\cdots,X'_m\}$ is given by
\begin{align}
X_k'&=X^{\sum_{1\leq i\leq m}[b_{ik}]_{+} {\bf e}_i -{\bf e}_k}+X^{\sum_{1\leq i\leq m}[-b_{ik}]_{+} {\bf e}_i -{\bf e}_k},\nonumber\\
X_i'&=X_i ,\quad 1\leq i\leq m,\quad i\neq k,\nonumber
\end{align}
where $[a]_{+}:=\max\{0,a\}$  for $a \in \mathbb{Z}$.

Note that the mutation is an involution. Two quantum seeds  $(\Lambda, \tilde{B}, X)$ and  $(\Lambda', \tilde{B}', X')$ are
called {\em mutation-equivalent}, denoted by $(\Lambda, \tilde{B}, X)\sim(\Lambda', \tilde{B}', X')$, if they can be obtained from each other
by iterated mutations. The set $\{X'_i~|1\leq i\leq n\}$ is called the {\em cluster}. 
The elements in  each cluster are called the {\em quantum cluster variables}, and  the  {\em quantum cluster monomials} are those $X'^{\bf e}$ such that $\bf e \in (\mathbb{Z}_{\geq 0})^{m}$  and 
$e_i=0$ for $n+1\leq i\leq m$. The
elements in $\mathbb{P}:=\{X_i~|~n+1\leq i\leq m\}$ are called the {\em coefficients}. Denote by
$\ZZ\mathbb{P}$ the ring of  Laurent polynomials in the elements of $\mathbb{P}$ with coefficients in $\ZZ[\mathfrak{q}^{\pm\frac{1}{2}}]$. Then the
{\em quantum cluster algebra}
$\mathcal{A}_{\mathfrak{q}}(\Lambda,\tilde{B})$ is defined to be the
$\ZZ\mathbb{P}$-subalgebra of $\mathcal{F}_{\mathfrak{q}}$ generated by all
quantum cluster variables.
The $\mathbb{Z}$-linear bar-involution on $\mathcal{T}_{\mathfrak{q}}$ is defined by
$$\overline{\mathfrak{q}^{\frac{r}{2}}X^{\mathbf{e}}}=\mathfrak{q}^{-\frac{r}{2}}X^{\mathbf{e}},\ \   \text{ for } r\in \mathbb{Z}  \text { and }\mathbf{e}\in \mathbb{Z}^{m}.$$
It is straightforward to show that $\overline{XY} = \overline{Y}\ \overline{X}$ for all $X, Y \in \mathcal{T}_{\mathfrak{q}}$ and that 
quantum cluster monomials are
bar-invariant. The bar-involution on $\mathcal{A}_{\mathfrak{q}}(\Lambda,\tilde{B})$ can be induced naturally.

\subsection{Quantum cluster characters for  type $\widetilde{A}_{2n-1,1}$}

Now we consider the  compatible pair $(\Lambda,\tilde{B})$  as follows: when $n \geq 2$, let 
\begin{equation*}\label{matrix-Lambda}
\Lambda:=\begin{pmatrix}
0 & 1 & 0 &1& \cdots & 1& 0 & 1  \\

-1 & 0& 1 &0& \cdots & 0& 1 & 0  \\

0 & -1& 0 &1& \cdots & 1& 0 & 1  \\

-1 & 0& -1 &0& \cdots & 0& 1 & 0  \\

\vdots & \vdots & \vdots & \vdots & & \vdots& \vdots& \vdots \\

-1 & 0& -1 &0& \cdots & 0& 1 & 0  \\

0 & -1& 0 &-1& \cdots & -1& 0 & 1  \\

-1 & 0& -1 &0& \cdots & 0& -1 & 0  \\

\end{pmatrix}_{2n\times 2n}.
\end{equation*} 
and 
\begin{equation*}\label{matrix-B}
\tilde{B}=B:=\begin{pmatrix}
0 & 1 & 0 &0& \cdots & 0& 0 & 1  \\

-1 & 0& 1 &0& \cdots & 0& 0 & 0  \\

0 & -1& 0 &1& \cdots & 0& 0 & 0  \\

0 & 0& -1 &0& \cdots & 0& 0 & 0  \\

\vdots & \vdots & \vdots & \vdots & & \vdots& \vdots& \vdots \\

0 & 0& 0 &0& \cdots & 0& 1 & 0  \\

0 & 0& 0 &0& \cdots & -1& 0 & 1  \\

-1 & 0& 0 &0& \cdots & 0& -1 & 0  \\

\end{pmatrix}_{2n\times 2n}
\end{equation*} 
and when $n=1$, let  $\Lambda:=\left(\begin{array}{cc} 0 & 1\\ -1 &
0\end{array}\right)$ and $\tilde{B}=B:=\left(\begin{array}{cc} 0 & 2\\
-2 & 0\end{array}\right)$.

Note that $\operatorname{det} \tilde{B} \neq 0$ and ${\tilde{B}}^{tr}\Lambda=2 \mathbf{I}_{2n}$, where $\mathbf{I}_{2n}$ is the identity matrix of size $2n$.   According to Rupel~\cite{rupel-2},  one can define a valued quiver from the compatible pair
$(\Lambda, \tilde{B})$. Since the the skew-symmetrizable principal submatrix of $\tilde{B}$ is itself,  one can associate a valued quiver $(Q, {\bf{d}})$ to  $\tilde{B}$
where the valuation $d_i$ is the $i$-th diagonal entry of the matrix $D$ occurring in the compatibility condition for $(\Lambda, \tilde{B})$. For the given compatible pair, the quiver $Q$ is the affine quiver of type 
$\widetilde{A}_{2n-1,1}$:

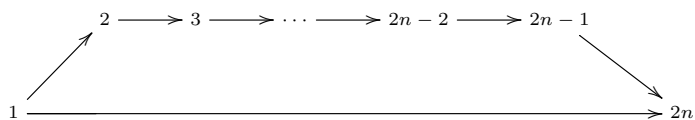
\begin{tiny}
\[ \xymatrix{ & 2  \ar[r]  & 3  \ar[r]  & \cdots  \ar[r]      & 2n-2    \ar[r] & 2n-1 \ar[dr] \\
 1 \ar@{-}[r] \ar[ru] \ar[rrrrrr]   &  & &  &  &  &   2n}\]
\end{tiny}
and the valuation vector is $\bf d$$=(2,2, \dots,2) \in \mathbb{Z}^{2n}$. 

In the rest of the paper,  we will work only on the valued quiver $(Q,\bf{d})$ of type $\widetilde{A}_{2n-1,1}.$ 

Let $\mathbb{F}$ be a finite field with cardinality $|\mathbb{F}|=q$ and $\overline{\mathbb{F}}$ an algebraic closure of $\mathbb{F}$. We denote by $\mathbb{F}_k$ the degree $k$ extension of $\mathbb{F}$ in $\overline{\mathbb{F}}$ for each positive integer $k$. The valued quiver   $(Q, {\bf{d}})$ used by Rupel in ~\cite{rupel} is essentially equivalent to the valued graph discussed by Dlab and Ringel in~\cite{DR}.   In a valued quiver,  the vertices are assigned values  rather than edges.  For the valued quiver $(Q, (2,2,\dots,2))$ with the set of  vertices $Q_0$ and the set of arrows $Q_1$, one can define its valued representation $V=\big(\{V_i\}_{1\leq i\leq 2n}, 
\{ \rho_{\alpha} \}_{\alpha \in Q_1}\big)$ by assigning an $\mathbb{F}_2$-vector space $V_i$ for each vertex $i$ and an $\mathbb{F}_2$-linear map $\rho_\alpha: V_{s(\alpha)} \rightarrow V_{t(\alpha)}$ for each arrow $\alpha: s(\alpha) \rightarrow t(\alpha)$ in $Q_1$. One can then define the category $\operatorname{Rep}_{\mathbb{F}} (Q, \bf d)$ of all finite dimensional valued representations of $(Q, \bf d)$. The dimension vector of a valued representation $V$ is defined by ${{\bf{dim}} }V=(\operatorname{dim}_{\mathbb{F}_2} V_i) _{1\leq i \leq 2n} \in \mathbb{Z}^{2n}$.  The simple representation at vertex $i$ is denoted by $S_i$ for $1\leq i \leq 2n$. The projective representation and injective representation of  $(Q, {\bf{d}})$ at vertex $i$ is denoted by $P_i$ and $I_i$ for $1\leq i \leq 2n$, respectively. A representation $V$ of $(Q, {\bf{d}})$ is called rigid if $\Ext^1 (V,V)=0$.

Let $R=(r_{ij})_{2n\times 2n} $ be the matrix with $$
r_{ij}:=\mathrm{dim}_{\End (S_{i})}\Ext^{1}(S_j,S_i),  \,\,\,\,\,\, 
 \text{ for }1 \leq i,j\leq 2n.$$ We can easily get that 
\begin{equation*}\label{matrix-R}
R=\begin{pmatrix}
0 & 0 & 0 &0& \cdots & 0& 0 & 0  \\

1 & 0& 0 &0& \cdots & 0& 0 & 0  \\

0 & 1& 0 &0& \cdots & 0& 0 & 0  \\

0 & 0& 1 &0& \cdots & 0& 0 & 0  \\

\vdots & \vdots & \vdots & \vdots & & \vdots& \vdots& \vdots \\

0 & 0& 0 &0& \cdots & 0& 0 & 0  \\

0 & 0& 0 &0& \cdots & 1& 0 & 0  \\

1 & 0& 0 &0& \cdots & 0& 1 & 0  \\

\end{pmatrix}.
\end{equation*} 
and  $\tilde{B}=R^{tr}-R$.

Let $\mathcal C_{Q}$ be the cluster category of
the valued quiver $(Q, \bf d)$ with the shift functor denoted by
$[1]$ and the Aulander-Reiten translation functor denoted by $\tau$.  Without causing any confusion, we denote by the same symbol $\tau$ the Auslander-Reiten translation functor on  $\operatorname{Rep}_{\mathbb{F}} (Q, \bf d)$.
 Each object $\widetilde{V}$ in $\mathcal C_Q$ can be uniquely
decomposed as
$$\widetilde{V}=V \oplus P[1],$$
where $V$ is a representation of $(Q,\bf d)$ and $P$ is a projective
representation of $(Q,\bf d)$. Now assume that $P=\bigoplus_{1\leq i \leq 2n}m_iP_i$  with $m_i \geq 0$. We can extend the
definition of the dimension vector $\bf dim$ of
representation of $(Q,\bf d)$ to object in $\mathcal C_Q$
by setting
$${\bf {dim}} \widetilde{V}={\bf {dim}} V-(m_i)_{1\leq i \leq 2n}.$$
 For any two (valued) representations $V$ and $W$ of $(Q,\bf d)$,  the Euler form is given by
$$\langle V,W\rangle=\mathrm{dim}_{\mathbb{F}}\operatorname{Hom}(V,W)-\mathrm{dim}_{\mathbb{F}}\operatorname{Ext}^{1}(V,W).$$
Note that the Euler form only depends on the dimension vectors of
$V$ and $W$ and therefore we can identify $\langle V,W\rangle$ with  $\langle \Dim V,\Dim W\rangle$ without causing any confusion. 
 The matrix associated with this form is
$(\mathbf{I}_{2n}-R^{tr})D=2(\mathbf{I}_{2n}-R^{tr})$ where $D=\operatorname{diag}(2,2,\dots,2)$. For simplicity we write $\mathbf{I}:=\mathbf{I}_{2n}$.

If $V$ is a representation of $(Q,\bf d)$ and $P$ is a projective representation of $(Q,\bf d)$,  the quantum cluster character is given by
\[X_{V\oplus P[1]}=\sum_{\bf{e}} |\mathrm{Gr}_{\bf{e}} V|q^{-\frac{1}{2}
\langle
\bf{e},\bf{v}-\bf{e}\rangle}X^{-{B}{\bf e}-({\mathbf{I}}-{R}^{tr}){\bf v}+{\bf t}_P},\]
where ${{\bf v}}={{\bf {dim}}} V$, ${\bf t}_P=\Dim (P/\operatorname{rad} P)$ and
$\mathrm{Gr}_{\bf{e}}V$ denotes the set of all subepresentations $U$
of $V$ with $\Dim U= \bf{e}$.  Note
that
$$
X_{P[1]}=X_{\tau P}=X^{\Dim (P/\rad
P)}=X^{\Dim \soc I}=X_{I[-1]}=X_{\tau^{-1}I}.
$$
for any projective representation   $P$ and
injective representation  $I$ of $(Q,\bf d)$  with $\soc I=P/\rad
P.$ In particular, $X_{P_i[1]}=X_i$ for $ 1\leq i \leq 2n$.

In the following, for convention, we define $\mathcal{A}_q(Q):= \mathcal{A}_{|\mathbb{F}|}(\Lambda, \tilde{B})$ as the specialization of $\mathcal{A}_{\mathfrak{q}}(\Lambda,\tilde{B})$ by setting $\mathfrak{q}=q=|\mathbb{F}|$.  The bar-involution on $\mathcal{A}_q(Q)$ is defined by $$\overline{q^{\frac{r}{2}}X^{\mathbf{e}}}=q^{-\frac{r}{2}}X^{\mathbf{e}},\ \   \text{ for } r\in \mathbb{Z}  \text { and }\mathbf{e}\in \mathbb{Z}^{m}$$
which is naturally induced by the bar-involution on  $\mathcal{A}_{\mathfrak{q}}(\Lambda,\tilde{B})$. 
\begin{theorem}\cite{fanqin,rupel-2}\label{bijection}
The map sending an object $V$ to $X_V$ induces a bijection from the set of isomorphism classes of rigid objects in $\mathcal C_Q$ to the set of cluster monomials of $\mathcal{A}_q(Q)$.
 \end{theorem}

Let $V$ and $W$ be representations of $(Q,\bf d)$.  From a morphism $\theta:W\to \tau V,$ we can get an exact sequence
 \begin{equation}
   \begin{tikzpicture}
   \matrix (m) [matrix of math nodes, row sep=3em, column sep=2.5em, text height=1.5ex, text depth=0.25ex]
    {0 & D & W & \tau V & \tau A \oplus I & 0\\};
   \path[->,font=\scriptsize]
    (m-1-1) edge (m-1-2)
    (m-1-2) edge (m-1-3)
    (m-1-3) edge node[auto] {$\theta$}(m-1-4)
    (m-1-4) edge (m-1-5)
    (m-1-5) edge (m-1-6);
  \end{tikzpicture}
 \end{equation}
 where $D=\operatorname{ker} \theta$, $\tau A \oplus I=\operatorname{coker} \theta$, $I$ is injective, $A$ and $V$ have the same maximal projective summand.  

\begin{theorem}\cite{dx-2, rupel-2}\label{m-1}
Assume  $V$ and $W$ are representations of $(Q,\bf d)$ with a unique (up to scalar) nontrivial extension $E\in\Ext^1(V,W)$, in particular $\dim_{\End(V)}\Ext^1(V,W)$$=1$.  Let $\theta\in\Hom(W,\tau V)$ be the equivalent morphism with $A, D, I$ described as above.  Furthermore assume that $\Hom(A\oplus D,I)=0=\Ext^1(A,D)$.  Then we have that
 \begin{align*}
 X_VX_W  = &q^{\half\Lambda((\mathbf{I}-R^{tr}){\Dim V},(\mathbf{I}-R^{tr}){\Dim W})}X_E\\
& + q^{\half\Lambda((\mathbf{I}-R^{tr}){\Dim V},(\mathbf{I}-R^{tr}){\Dim W})+\half\langle{V},{W}\rangle-\half\langle {A}, {D}\rangle}X_{D\oplus A\oplus I[-1]}.
\end{align*} 
 \end{theorem}
 
 Let $W$ be a  representation and $I$ an injective representation of $(Q, {\bf d})$.  Let $\nu$ be the Nakayama functor on $\operatorname{Rep}_{\mathbb{F}} (Q, \bf d)$. 
 Write $P=\nu^{-1}(I)$ and note that $P$ is projective with $\soc I\cong P/\rad P$ and  $\End(I)\cong \End(P)$.  From morphisms $\theta: W\to I$ and $\gamma:P\to W$ we get exact sequences
 \begin{center}
  \begin{tikzpicture}
   \matrix (m) [matrix of math nodes, row sep=1em, column sep=2.5em, text height=1.5ex, text depth=0.25ex]
    {0 & G & W & I & I' & 0\,,\\ 0 & P' & P & W & F & 0\,,\\};
   \path[->,font=\scriptsize]
    (m-1-1) edge (m-1-2)
    (m-1-2) edge (m-1-3)
    (m-1-3) edge node[auto] {$\theta$} (m-1-4)
    (m-1-4) edge (m-1-5)
    (m-1-5) edge (m-1-6)
    (m-2-1) edge (m-2-2)
    (m-2-2) edge (m-2-3)
    (m-2-3) edge node[auto] {$\gamma$} (m-2-4)
    (m-2-4) edge (m-2-5)
    (m-2-5) edge (m-2-6);
  \end{tikzpicture}
 \end{center}
 where $G=\ker\theta$, $I'=\operatorname{coker}\theta$ is injective, $P'=\ker\gamma$ is projective, and $F=\operatorname{coker}\gamma$.

\begin{theorem}\cite{dx-2,rupel-2}\label{m-2}
Let $W$, $I$ and $P$  be the  representations of $(Q, {\bf d})$ defined as above.  Assume that there exist unique (up to scalar) morphisms $f\in\Hom(W,I)$ and $g\in\Hom(P,W)$, in particular $\dim_{\End(I)}\Hom(W,I)=\dim_{\End(P)}\Hom(P,W)=1$.  Define $F$, $G$, $I'$, $P'$ as above and assume further that $\Hom(P',F)=\Hom(G,I')=0$.  Then we have
\begin{align*}
X_WX_{I[-1]} =& q^{-\half\Lambda((\mathbf{I}-R^{tr}){\Dim W},(\mathbf{I}-R^{tr}){\Dim I})}X_{G\oplus I'[-1]}\\
& + q^{-\half\Lambda((\mathbf{I}-R^{tr}){\Dim W},(\mathbf{I}-R^{tr}){\Dim I})-\half\mathrm{dim}_{\mathbb{F}}\operatorname{End}(I)}X_{F\oplus P'[1]}.
\end{align*}   
\end{theorem}

\begin{theorem}\cite{DShCh,rupel-2}\label{m-3}
  Let $V$ and $W$ be representations of $(Q,\bf d)$ with $\Ext^1(V,W)=0$, then we have
  \[X_VX_W=q^{\half\Lambda((\mathbf{I}-R^{tr}){\Dim V},(\mathbf{I}-R^{tr}){\Dim W})}X_{V\oplus W}.\]
  In addition if $\Ext^1(W,V)=0$, then we have
  \[X_VX_W = q^{\Lambda((\mathbf{I}-R^{tr}){\Dim V},(\mathbf{I}-R^{tr}){\Dim W})}X_WX_V.\]
 \end{theorem}

 \begin{theorem}\cite{DShCh,rupel-2}\label{th:init-comm}
  Let $V$ and $I$ be representations of $(Q,\bf d)$ such that $I$ is injective and $V\oplus I[-1]$ is rigid, then we have
  \[X_VX^{(\mathbf{I}-R^{tr})\Dim I} = q^{-\half\Lambda((\mathbf{I}-R^{tr}){\Dim V}, (\mathbf{I}-R^{tr})\Dim I)} X_{V\oplus I[-1]}\]
  and
  \[ X_VX^{(\mathbf{I}-R^{tr})\Dim I} =q^{-\Lambda((\mathbf{I}-R^{tr}){\Dim V}, (\mathbf{I}-R^{tr})\Dim I)} X^{(\mathbf{I}-R^{tr})\Dim I} X_V\]
 \end{theorem}

\section{Quantum linear relations in  quantum cluster algebra $\mathcal{A}_q(Q)$}
In this section, we study the quantum cluster algebra of type $\widetilde{A}_{2n-1,1}$ for $n\geq 2$.
Due to the same valuations in ${\bf d}=(2,2, \dots,2)$,  every representation of $\operatorname{Rep}_{\mathbb{F}} (Q, \bf d)$ is just the representation of the (unvalued) quiver $\widetilde{A}_{2n-1,1}$ over the finite field $\mathbb{F}_2$.
Thus the regular components of the Auslander-Reiten quiver of representations of $(Q, {\bf d})$
consist of one non-homogeneous tube of rank $2n-1$ and a family of homogeneous tubes indexed by projective line over $\mathbb{F}_2$. The mouth of homogeneous tube has the form:
\begin{equation*}
E(\lambda):=
\begin{tiny}
\begin{tikzcd}  
& \mathbb{F}_2 \ar[r, "1"]  &  \mathbb{F}_2 \ar[r, "1"]  & \cdots    \ar[r, "1"]    & \mathbb{F}_2    \ar[r, "1"] & \mathbb{F}_2  \ar[dr, "1"]  & \\
  \mathbb{F}_2 \arrow[ur, "1"]   \arrow[rrrrrr, "\lambda"]  & & &   &   &  & \mathbb{F}_2
\end{tikzcd}
\end{tiny}
\end{equation*} 
with $\lambda \in \mathbb{F}_2$. We denote by $\delta:={\Dim} E(\lambda)$.
The mouth of the unique non-homogeneous tube has the following shape:
\begin{center}
\begin{tiny}
\begin{tikzcd}  
                             & \star  \arrow[rd]  &                                      &     \star   \arrow[rd] &                      &              &                               &  \star  \arrow[rd]       &                         &   \star \arrow[rd] &\\
  E_1  \arrow[ur]   &                            & E_{2n-1}   \arrow[ur]     &                               &  E_{2n-2}      & \cdots   &   E_3  \arrow[ur]      &                                &  E_2  \arrow[ur]   &            & E_1 .
   \end{tikzcd}
\end{tiny}
\end{center}
where 
\begin{equation*}
E_1=
\begin{tiny}
\begin{tikzcd}  
& 0 \ar[r]  &  0 \ar[r]  & \cdots    \ar[r]    & 0    \ar[r] & 0 \ar[dr]  & \\
  \mathbb{F}_2 \arrow[ur]   \arrow[rrrrrr, "1"]  & & &   &   &  & \mathbb{F}_2
\end{tikzcd}
\end{tiny},
\end{equation*}
$E_i=S_i$ for $2\leq i \leq 2n-1$, $\tau (E_i)=E_{i+1}$ for $1\leq i \leq 2n-2$, and $\tau (E_{2n-1})=E_1$.  Given a regular simple module $S$ in a tube, we denote by $R_{S,l}$  the indecomposable regular module with quasi-socle $S$ and quasi-length $l$ for any $l\geq 0.$ 

The Auslander-Reiten component of preprojective representations in $\operatorname{Rep}_{\mathbb{F}} (Q, \bf d)$ has the following shape:

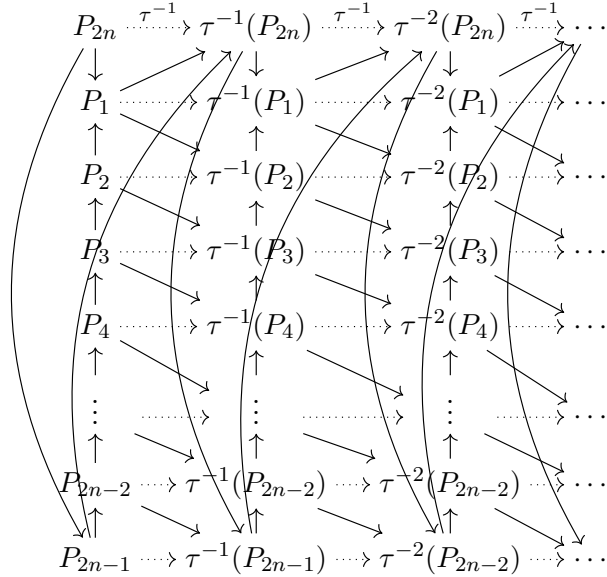
\begin{center}
\begin{tikzcd}[cramped, sep=small]
  P_{2n} \arrow[r,  "\tau^{-1}", dotted] \arrow[d]   \arrow[ddddddd, bend right]  &  \tau^{-1} (P_{2n})  \arrow[d]   \arrow[r, "\tau^{-1}", dotted]  \arrow[ddddddd, bend right]           &   \tau^{-2} (P_{2n}) \arrow[d]  \arrow[ddddddd, bend right]   \arrow[r, "\tau^{-1}", dotted]   & \cdots    \arrow[ddddddd, bend right] \\
  P_1  \arrow[r, dotted]    \arrow[dr]   \arrow[ur]                          &    \tau^{-1} (P_{1})  \arrow[r, dotted]    \arrow[ur]   \arrow[dr] &\tau^{-2} (P_{1}) \arrow[r, dotted]   \arrow[ur]  \arrow[dr] & \cdots                               \\
  P_2     \arrow[u] \arrow[r, dotted]  \arrow[dr]  &  \tau^{-1} (P_{2})  \arrow[r, dotted] \arrow[u]  \arrow[dr]  & \tau^{-2} (P_{2})  \arrow[u]           \arrow[r, dotted]    \arrow[dr]  & \cdots       \\
  P_3   \arrow[u] \arrow[r, dotted] \arrow[dr] &  \tau^{-1} (P_{3}) \arrow[r, dotted] \arrow[u] \arrow[dr]  &  \tau^{-2} (P_{3})  \arrow[u]       \arrow[r, dotted]   \arrow[dr] & \cdots          \\
  P_4   \arrow[u] \arrow[r, dotted] \arrow[dr] &  \tau^{-1} (P_{4})  \arrow[r, dotted] \arrow[u]\arrow[dr] &  \tau^{-2} (P_{4}) \arrow[u]           \arrow[r, dotted]   \arrow[dr] & \cdots       \\
\,\,\,\,\,\,\, \vdots \,\,\,\,\,\,\, \arrow[u] \arrow[r, dotted] \arrow[dr] &\,\,\,\, \,\,\,\vdots \,\, \,\,\,\,\, \arrow[r, dotted] \arrow[u] \arrow[dr] &\,\,\,\,\,\,\, \vdots \,\,\,\,\,\,\, \arrow[u]      \arrow[r, dotted]    \arrow[dr]          & \cdots                         \\  
 P_{2n-2}  \arrow[u] \arrow[r, dotted] \arrow[dr] &  \tau^{-1} (P_{2n-2})  \arrow[r, dotted] \arrow[u] \arrow[dr] &   \tau^{-2} (P_{2n-2}) \arrow[u]   \arrow[r, dotted]  \arrow[dr]  & \cdots       \\
 P_{2n-1}   \arrow[u] \arrow[r, dotted]  \arrow[uuuuuuur, bend left] &  \tau^{-1} (P_{2n-1}) \arrow[r, dotted] \arrow[u]  \arrow[uuuuuuur, bend left] &  \tau^{-2} (P_{2n-2}) \arrow[u]   \arrow[r, dotted]   \arrow[uuuuuuur, bend left] & \cdots       \\
    \end{tikzcd}
\end{center}
and the Auslander-Reiten component of preinjective representations has the following shape:

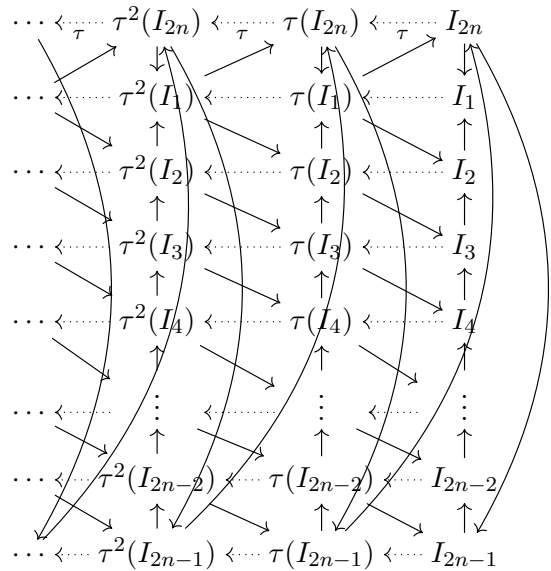
\begin{center}
\begin{tikzcd}[cramped, sep=small]
  & \cdots  \arrow[ddddddd, bend left]  &  \tau^2 (I_{2n})  \arrow[d]   \arrow[l, "\tau", dotted]  \arrow[ddddddd, bend left]           &  \tau (I_{2n}) \arrow[d]  \arrow[ddddddd, bend left]   \arrow[l,  "\tau", dotted]   & I_{2n}  \arrow[l,  "\tau", dotted] \arrow[d]  \arrow[ddddddd, bend left] \\
 &  \cdots      \arrow[dr]   \arrow[ur]                          &   \tau^2 (I_{1})       \arrow[l, dotted]    \arrow[ur]   \arrow[dr] &   \tau (I_{1})     \arrow[l, dotted]   \arrow[ur]  \arrow[dr] & I_{1}    \arrow[l, dotted]                           \\
 & \cdots     \arrow[dr]  &  \tau^2 (I_{2})  \arrow[l, dotted] \arrow[u]  \arrow[dr]  &    \tau (I_{2})       \arrow[u]           \arrow[l, dotted]    \arrow[dr]  & I_{2}  \arrow[l, dotted]   \arrow[u]    \\
 & \cdots    \arrow[dr] &  \tau^2 (I_{3})  \arrow[l, dotted] \arrow[u] \arrow[dr]  &  \tau (I_{3})  \arrow[u]       \arrow[l, dotted]   \arrow[dr] & I_{3}  \arrow[l, dotted]   \arrow[u]    \\
  & \cdots  \arrow[dr] &  \tau^2 (I_{4})  \arrow[l, dotted] \arrow[u]\arrow[dr] &  \tau (I_{4})  \arrow[u]           \arrow[l, dotted]   \arrow[dr] & I_{4} \arrow[l, dotted] \arrow[u] \\
&  \cdots  \arrow[dr] &\,\,\,\, \,\,\,\vdots \,\, \,\,\,\,\, \arrow[l, dotted] \arrow[u] \arrow[dr] &\,\,\,\,\,\,\, \vdots \,\,\,\,\,\,\, \arrow[u]      \arrow[l, dotted]   \arrow[dr]  &  \,\,\,\,\,\,\, \vdots \,\,\,\,\,\,\, \arrow[l, dotted] \arrow[u]\\  
 & \cdots   \arrow[dr] & \tau^2 (I_{2n-2})   \arrow[l, dotted] \arrow[u] \arrow[dr] &  \tau (I_{2n-2})  \arrow[u]   \arrow[l, dotted]  \arrow[dr]  & I_{2n-2} \arrow[l, dotted] \arrow[u]\\
 & \cdots    \arrow[uuuuuuur, bend right] &  \tau^2 (I_{2n-1}) \arrow[l, dotted] \arrow[u]  \arrow[uuuuuuur, bend right] &  \tau (I_{2n-1})  \arrow[u]   \arrow[l, dotted]   \arrow[uuuuuuur, bend right] & I_{2n-1} \arrow[l, dotted] \arrow[u]\\
    \end{tikzcd}
\end{center}
For $1 \leq i \leq 2n-1$ and $l \geq 0$,  let $M_i(l)$ be the indecomposable preprojective representation such that $\Dim M_i(l)=\Dim P_{i+1}+l\delta$.
For $2 \leq i \leq 2n$ and $l \geq 0$, let $N_i(l)$ be the indecomposable preinjective representation such that $\Dim N_i(l)=\Dim I_{i-1}+l\delta$.  

\begin{proposition} \label{delta-1}
For any $\lambda \in \mathbb{F}_2$, we have $\overline{X_{E(\lambda)}}=X_{E(\lambda)}.$ 
\end{proposition}
\begin{proof}
Note that the subrepresentation of $E(\lambda)$ has the form 
\begin{center}
\begin{tiny}
\begin{tikzcd}  
& 0 \ar[r]  &  0 \ar[r]  & \cdots    \ar[r]    & \mathbb{F}_2    \ar[r, "1"] & \mathbb{F}_2 \ar[dr, "1"]  & \\
  0 \arrow[ur]   \arrow[rrrrrr]  & & &   &   &  & \mathbb{F}_2
\end{tikzcd}
\end{tiny}.
\end{center}
Let ${\bf e}$ be the dimension vector of the above subrepresentation. Then we have that $$\langle {\bf{e}}, {\Dim} E(\lambda)-{\bf{e}} \rangle =2{\bf{e}}^{tr}(\mathbf{I}-R^{tr})({\Dim} E(\lambda)-{\bf{e}} )=0, $$ which implies
that 
$X_{E(\lambda)}=\displaystyle\sum_{\bf{e}}  X^{-{B}{\bf e}-(\mathbf{I}-{R}^{tr})  {\Dim} E(\lambda)}$. Therefore $\overline{X_{E(\lambda)}}=X_{E(\lambda)}.$ 
\end{proof}

Note that, for any $\lambda \in \mathbb{F}_2$, $X_{E(\lambda)}=\displaystyle\sum_{\bf{e}}  X^{-{B}{\bf e}-(\mathbf{I}-{R}^{tr})  \delta}$. Therefore we can define a generic element $X_{\delta}:=X_{E(\lambda)}$ which is independent of the choice of $\lambda$.

\begin{proposition}\label{delta-2}
With the above notations, we have that
$$X_{I_1} X_{P_2}=\displaystyle q^{\frac{1}{2}} X_{\delta}+ q^{-1} X_{I_{2n}[-1]}X_{P_3}.$$
\end{proposition}
\begin{proof}
Note that
$$X_{I_1}=X^{-{B}{\Dim} I_1-(\mathbf{I}-{R}^{tr})  {\Dim} I_1}+X^{-(\mathbf{I}-{R}^{tr})  {\Dim} I_1},$$
$$X_{P_2}=X^{-({I}-{R}^{tr})  {\Dim} P_2}+\sum^{2n}_{i=2}  X^{-B {\Dim}  P_i-(\mathbf{I}-{R}^{tr}) {\Dim} P_2},$$
and
$$X_{P_3}=X^{-(\mathbf{I}-{R}^{tr})  {\Dim} P_3}+\sum^{2n}_{i=3}  X^{-B {\Dim}  P_i-(\mathbf{I}-{R}^{tr}) {\Dim} P_3}.$$
Thus the statement can be proved by  the generic property of $X_\delta$ and direct computation.
\end{proof}
Immediately, we obtain the following result.
\begin{corollary}\label{delta-3}
 The generic element $X_{\delta}$ belongs to  $\mathcal{A}_q(Q)$. 
\end{corollary}
\begin{lemma}\label{lamda}
With the above notations, we have that
\begin{enumerate}
\item $\Lambda ((\mathbf{I}-R^{tr})\Dim M_1(l), (\mathbf{I}-R^{tr})\delta)=-1$;
\item $\Lambda ((\mathbf{I}-R^{tr})\Dim E, (\mathbf{I}-R^{tr})\delta)=0$ for  any regular rigid object  $E$ in $\mathcal{C}_{Q}$.
\end{enumerate}

\end{lemma}
\begin{proof}
(1) The equation follows from $$\Lambda ((\mathbf{I}-R^{tr})\Dim M_1(l), (\mathbf{I}-R^{tr})\delta)= \Lambda ((\mathbf{I}-R^{tr})\Dim P_2, (\mathbf{I}-R^{tr})\delta)=-1.$$
(2) For any regular rigid object $E$ in $\mathcal{C}_{Q}$, there exist some $n_i\in \mathbb{Z}_{\geq 0}$, $1\leq i\leq 2n-1$ such that $\Dim E=\displaystyle \sum^{2n-1}_{i=1}n_i\Dim E_i$.  Note that for $i=1$, we have
$$\Lambda ((\mathbf{I}-R^{tr})\Dim E_1, (\mathbf{I}-R^{tr})\delta)=\Lambda(-e_{2n-1}+e_{2n},-e_{1}+e_{2n})=0$$
and for $2\leq i\leq 2n-1,$ we have
$$\Lambda ((\mathbf{I}-R^{tr})\Dim E_i, (\mathbf{I}-R^{tr})\delta)=\Lambda(-e_{i-1}+e_{i},-e_{1}+e_{2n})=0.$$
Thus the second equation follows from  summing up the above equations.

\end{proof}

\begin{proposition}\label{bar-inv}
For any regular rigid object $E$ in $\mathcal{C}_{Q}$ and any $i\in \mathbb{Z}_{\geq 0}$, the element $X^{i}_{\delta}X_{E}$ is bar-invariant.
\end{proposition}
\begin{proof}
Note that $\dim\Ext^1(E(\lambda),E)=\dim\Ext^1(E,E(\lambda))=0,$ then by Theorem \ref{m-3} and Lemma \ref{lamda}, we have
$$X_{E(\lambda)}X_{E}=q^{\frac{1}{2} \Lambda ((\mathbf{I}-R^{tr})\delta, (\mathbf{I}-R^{tr}) \Dim E)}X_{E(\lambda)\oplus E}=X_{E(\lambda)\oplus E}$$
and
$$X_{E}X_{E(\lambda)}=q^{\frac{1}{2} \Lambda ((\mathbf{I}-R^{tr}) \Dim E),(\mathbf{I}-R^{tr})\delta}X_{E(\lambda)\oplus E}=X_{E\oplus E(\lambda)}.$$
Thus $X_{E(\lambda)}X_{E}=X_{E}X_{E(\lambda)},$ and then we have
$$\overline{X^{i}_{E(\lambda)}X_{E}}=\overline{X_{E}}(\overline{X_{E(\lambda)}})^{i}=X_{E}X^{i}_{E(\lambda)}=X^{i}_{E(\lambda)}X_{E}.$$\end{proof}

\begin{lemma}\label{induc1}
With the above notations, we have
\begin{enumerate}
\item $X_\delta X_{I_1[-1]}=\displaystyle q^{\frac{1}{2}} X_{M_1(0)}+ q^{-\frac{1}{2}} X_{I_{2n}[-1]}$;
\item $X_\delta X_{M_1(0)}=\displaystyle q^{\frac{1}{2}} X_{M_1(1)}+ q^{-\frac{1}{2}} X_{I_{1}[-1]}$;
\item $X_\delta X_{M_1(l)}=\displaystyle q^{\frac{1}{2}} X_{M_1(l+1)}+ q^{-\frac{1}{2}} X_{M_1(l-1)}$ for $l \geq 1$.
\end{enumerate}
\end{lemma}
\begin{proof}
(1) For any $\lambda \in \mathbb{F}_2$, we have short exact sequences $$0 \longrightarrow P_2 \longrightarrow E(\lambda) \longrightarrow I_1 \longrightarrow 0 \text{ and }
 0 \longrightarrow P_{2n} \longrightarrow P_1  \longrightarrow E(\lambda) \longrightarrow 0.$$  Then by Theorem \ref{m-2}, we have
\begin{align*}
&X_\delta X_{I_1[-1]}  \\
=& \displaystyle q^{-\frac{1}{2} \Lambda ((\mathbf{I}-R^{tr})\delta, (\mathbf{I}-R^{tr})   \Dim I_1)}X_{P_2} +  q^{-\frac{1}{2}  \Lambda ((\mathbf{I}-R^{tr})\delta, (\mathbf{I}-R^{tr}) \Dim I_1)-\frac{1}{2}\dim_{\mathbb{F}} \End (I_1)}  X_{P_{2n}[1]} \\
=& \displaystyle q^{\frac{1}{2}} X_{M_1(0)}+ q^{-\frac{1}{2}} X_{I_{2n}[-1]}.
\end{align*}
(2)  For any $\lambda \in \mathbb{F}_2$, we also have short exact sequences  $$0 \longrightarrow M_1(0) \longrightarrow M_1(1)  \longrightarrow E(\lambda) \longrightarrow 0 \text{ and }
 0 \longrightarrow M_1(0) \longrightarrow   E(\lambda)     \longrightarrow I_1 \longrightarrow 0.$$
 Then by Theorem \ref{m-1}, we have 
\begin{align*}
&X_\delta X_{M_1(0)} \\  = &\displaystyle q^{\frac{1}{2} \Lambda ((\mathbf{I}-R^{tr})\delta, (\mathbf{I}-R^{tr}) \Dim M_1(0))}X_{M_1(1)} \\
&+  q^{\frac{1}{2}  \Lambda ((\mathbf{I}-R^{tr})\delta, (\mathbf{I}-R^{tr}) \Dim M_1(0)) +\frac{1}{2}\langle \delta, \Dim M_1(0) \rangle }  X_{I_{1}[-1]} \\
=& \displaystyle q^{\frac{1}{2}} X_{M_1(1)}+ q^{-\frac{1}{2}} X_{I_{1}[-1]}.
\end{align*}
(3) For $l \geq 1$, from the short exact sequences 
$$0 \longrightarrow M_1(l) \longrightarrow M_1(l+1)  \longrightarrow E(\lambda) \longrightarrow 0 \text{ and }
 0 \longrightarrow M_1(l-1) \longrightarrow  M_1(l)      \longrightarrow E(\lambda) \longrightarrow 0,$$
and by Theorem \ref{m-1}, we have 
\begin{align*}
&X_\delta X_{M_1(l)} \\  = &\displaystyle q^{\frac{1}{2} \Lambda ((\mathbf{I}-R^{tr})\delta, (\mathbf{I}-R^{tr}) \Dim M_1(l))}X_{M_1(l+1)} \\ 
&+  q^{\frac{1}{2}  \Lambda ((\mathbf{I}-R^{tr})\delta, (\mathbf{I}-R^{tr}) \Dim M_1(l)) +\frac{1}{2}\langle \delta, \Dim M_1(l) \rangle }  X_{M_1(l-1)} \\
=& \displaystyle q^{\frac{1}{2}} X_{M_1(l+1)}+ q^{-\frac{1}{2}} X_{M_{1}(l-1)}. \end{align*}
The proof is completed.
\end{proof}

Similarly, we obtain the following results.
\begin{lemma}\label{induc11}
With the above notations, we have that
\begin{enumerate}
\item $ X_{I_{2n}[-1]}X_\delta =\displaystyle q^{\frac{1}{2}} X_{I_{2n-1}}+ q^{-\frac{1}{2}} X_{I_{1}[-1]}$;
\item $ X_{I_{2n-1}} X_\delta =\displaystyle q^{\frac{1}{2}} X_{N_{2n}(1)}+ q^{-\frac{1}{2}} X_{I_{2n}[-1]}$;
\item $ X_{N_{2n}(l)} X_\delta =\displaystyle q^{\frac{1}{2}} X_{N_{2n}(l+1)}+ q^{-\frac{1}{2}} X_{N_{2n}(l-1)}$ for $l \geq 1$.
\end{enumerate}
\end{lemma}

Recall that the well-known  $m$-th Chebyshev polynomial of the first kind $F_{m}(x)$ is  defined by
  $$
  F_0(x)=1,F_1(x)=x, F_2(x)=x^2-2, \text{ and } F_{m+1}(x)=F_{m}(x)x-F_{m-1}(x)~\text{for}~m\geq2.
  $$

\begin{theorem}\label{mainresult1} In $\mathcal{A}_q (Q)$, for $m \geq 1$ we have that 
\begin{enumerate}
\item $F_m(X_\delta) X_{I_1[-1]}= \displaystyle q^{\frac{m}{2}} X_{M_1(m-1)}+ q^{-\frac{m}{2}} X_{N_{2n}(m-2)}$,
where $N_{2n}(-1)=I_{2n}[-1]$;
\item $F_m(X_\delta) X_{I_i[-1]}= \displaystyle q^{\frac{m}{2}} X_{M_i(m-1)}+ q^{-\frac{m}{2}} X_{N_i(m-1)}$  for $2 \leq i \leq 2n-1$;
\item  $F_m(X_\delta) X_{I_{2n}[-1]}= \displaystyle q^{\frac{m}{2}} X_{M_1(m-2)}+ q^{-\frac{m}{2}} X_{N_{2n}(m-1)}$.
\end{enumerate}
\end{theorem}
\begin{proof}
We only prove statement (1). The proof of others are similar.
We will prove the statement by induction on $m$. If $m=1$, the statement follows from Lemma~\ref{induc1} immediately.  If $m=2$, by Lemma~\ref{induc1} and applying bar-involution on equation (1) in Lemma ~\ref{induc11},   we have that
\begin{align*}
& F_2(X_{\delta}) X_{I_1[-1]} = X_\delta^2 X_{I_1[-1]}-2 X_{I_1[-1]}=X_\delta ( \displaystyle q^{\frac{1}{2}} X_{M_1(0)}+ q^{-\frac{1}{2}} X_{I_{2n}[-1]})- 2 X_{I_1[-1]}\\
=& \displaystyle q^{\frac{1}{2}} ( \displaystyle q^{\frac{1}{2}} X_{M_1(1)}+ q^{-\frac{1}{2}} X_{I_{1}[-1]})+   q^{-\frac{1}{2}}(q^{\frac{1}{2}} X_{I_1[-1]}+ q^{-\frac{1}{2}} X_{I_{2n-1}}) -2X_{I_{1}[-1]}\\
=& q X_{M_1(1)} +q^{-1} X_{I_{2n-1}}=q X_{M_1(1)} +q^{-1} X_{N_{2n}(0)}.
\end{align*}

Assume that the statement is true for some $m \geq 3$. Then for $m+1$, by Lemma~\ref{induc1} and applying bar-involution on equation (3) in Lemma ~\ref{induc11}, and the definition of $F_{m}(x)$, we have that
\begin{align*}
& F_{m+1}(X_{\delta}) X_{I_1[-1]} = (X_\delta F_m (X_\delta)-F_{m-1}(X_\delta) ) X_{I_1[-1]} \\
=&X_\delta ( \displaystyle q^{\frac{m}{2}} X_{M_1(m-1)}+ q^{-\frac{m}{2}} X_{N_{2n}(m-2)})-  ( q^{\frac{m-1}{2}} X_{M_1(m-2)}+ q^{-\frac{m-1}{2}} X_{N_{2n}(m-3)})               \\
=&  \displaystyle q^{\frac{m}{2}}         ( \displaystyle q^{\frac{1}{2}} X_{M_1(m)}+ q^{-\frac{1}{2}} X_{M_1(m-2)})+ q^{-\frac{m}{2}}   ( q^{-\frac{1}{2}} X_{N_{2n}(m-1)}+ q^{\frac{1}{2}} X_{N_{2n}(m-3)})               \\
& -  ( q^{\frac{m-1}{2}} X_{M_1(m-2)}+ q^{-\frac{m-1}{2}} X_{N_{2n}(m-3)})               \\
=& q^{\frac{m+1}{2}} X_{M_1(m)}+ q^{-\frac{m+1}{2}} X_{N_{2n}(m-1)}.
\displaystyle
\end{align*}
The proof is completed.
\end{proof}

Note that the shift functor $[1]$ on $\mathcal{C}_Q$ induces an automorphism of $\mathcal{A}_q(Q)$ denoted by $\sigma$ such that
$\sigma (X_M)=X_{M[1]}$ for any indecomposable rigid object $M$ of $\mathcal{C}_Q$ and $\sigma (q^{\frac{m}{2}})=q^{\frac{m}{2}}$ for any $m \in \mathbb{Z}$.
For the reader's convenience, we provide the proof of this statement in the appendix. 
 
\begin{proposition}\label{delta-4}
With the above notations, we have that $\sigma(X_\delta)=X_\delta$.
\end{proposition}
\begin{proof}
For any $\lambda \in \mathbb{F}_2$, we have short exact sequences $$0 \longrightarrow E(\lambda) \longrightarrow \tau (I_{2n-1}) \longrightarrow I_{2n} \longrightarrow 0 \text{ and }
 0 \longrightarrow E(\lambda) \longrightarrow \tau (I_{2n}) \longrightarrow \tau (I_{1}) \longrightarrow 0.$$  Then by Theorem \ref{m-1}, we have
$$X_{I_{2n}}X_\delta 
= q^{\frac{1}{2}} X_{\tau (I_{2n-1}) }+ q^{-\frac{1}{2}} X_{I_1}.$$
Apply $\sigma$ to the equation (1) in Lemma \ref{induc11}, we have
$$X_{I_{2n}}\sigma(X_\delta) 
= q^{\frac{1}{2}} X_{\tau  (I_{2n-1})}+ q^{-\frac{1}{2}} X_{I_1}.$$ 
Then we obtain $\sigma(X_\delta)=X_\delta$.
\end{proof}

The following result can be viewed as  the quantum analogue of the constant coefficient linear relations.
\begin{theorem} \label{linear}
For $1 \leq i \leq 2n$, $m \in \mathbb{Z}$, we have that $$F_{2n}(X_\delta) X_{\tau^{m}(I_{i}})= \displaystyle q^{n} 
X_{\tau^{m-2n+3}(P_{i})} + q^{-n} X_{\tau^{m+2n-1}(I_{i})}.$$
\end{theorem}
\begin{proof}
According to Theorem \ref{mainresult1}, we have
$$F_{2n}(X_\delta) X_{I_{i}[-1]}= \displaystyle q^{n} X_{\tau^{2-2n}(P_{i})} + q^{-n} X_{\tau^{2n-2}(I_i)}.$$
By Proposition \ref{delta-4} and the definition of the $m$-th Chebyshev polynomial of the first kind $F_{m}(x)$, we have $\sigma(F_{2n}(X_\delta))=F_{2n}(X_\delta)$. By applying  $\sigma^{m+1}$ on the above equation, we can obtain the result immediately. 
\end{proof}

\section{Bases of the quantum cluster algebra $\mathcal{A}_q(Q)$}
In this section, we will construct two bar-invariant  bases of the quantum cluster algebra $\mathcal{A}_q(Q)$. These bases  contain quantum cluster monomials and consist of positive elements. 
\subsection{The bar-invariant atomic basis}
Now we consider the following set
$$
 \begin{aligned}
&\mathcal{B}=\{X_M | M \text{ is rigid in } \mathcal{C}_Q \} \cup \{F_{m}(X_\delta) X_E | m \geq 1, E \text{ is a regular rigid object in } \mathcal{C}_Q \}.\\
  \end{aligned}
$$
Note that specializing $q=1$, the set $\mathcal{B}$ is  the atomic basis of the corresponding classical cluster algebra (see \cite{DTh} for all type $\widetilde{A}_{p,q}$). 

\begin{definition}
An element $Y$ in $\mathcal{A}_q(Q)$ is called positive if the Laurent expansion of $Y$ in the cluster variables of any cluster of $\mathcal{A}_q(Q)$ has $\mathbb{Z}_{\geq 0}[q^{\pm \frac{1}{2}}]$-coefficients.
A $\mathbb{Z}[q^{\pm \frac{1}{2}}]$-basis of $\mathcal{A}_q(Q)$ is called the atomic basis if any positive element in $\mathcal{A}_q(Q)$ is a $\mathbb{Z}_{\geq 0}[q^{\pm \frac{1}{2}}]$-combination of these basis elements.
\end{definition}

\begin{lemma}\label{step-1}
 For any $m \in \mathbb{Z}$, in  $\mathcal{A}_q(Q)$ we have that
\begin{enumerate}
\item $X_{\tau^{m} (I_1)} X_{\tau^m (S_2)}=q^{ \frac{1}{2}} X_{\tau^{m} (I_2)}+ q^{-\frac{1}{2}} X_{\tau^{m-1} (I_{2n})}$;
\item $X_{\tau^{m} (I_l)} X_{\tau^m (S_{l+1})}=q^{ \frac{1}{2}} X_{\tau^{m} (I_{l+1})}+ q^{-\frac{1}{2}} 
X_{\tau^{m} (I_{l-1})}$ for $2 \leq l \leq 2n-2$;
\item  $X_{\tau^{m} (I_{2n-1})} X_{\tau^m (E_{1})}=q^{ \frac{1}{2}} X_{\tau^{m} (I_{2n})}+ q^{-\frac{1}{2}} X_{\tau^{m} (I_{2n-2})}$;
\item $X_{\tau^{m} (I_{2n})} X_{\tau^m (S_{2})}=q^{ \frac{1}{2}} X_{\tau^{m+1} (I_1)}+ q^{-\frac{1}{2}} X_{\tau^{m} (I_{2n-1})}$.
\end{enumerate}
\end{lemma}
\begin{proof}
We only prove statement (1). The proof of others is similar.

Note that  $\operatorname{dim}_{\operatorname{End}(I_1)}  \operatorname{Ext}(I_1, S_2)=1$, and the related short exact sequences are
\[ \xymatrix{0 \ar[r]& S_2  \ar[r] & I_2 \ar[r]& I_1 \ar[r]&0}\] and
\[ \xymatrix{0 \ar[r]& S_2  \ar[r] & \tau (I_1) \ar[r]& I_{2n} \ar[r]&0} \]
Then by Theorem \ref{m-1}, we can  get 
 $$X_{I_1} X_{S_2}=q^{ \frac{1}{2}} X_{I_2}+ q^{-\frac{1}{2}} X_{I_{2n}[-1]}.$$
 The statement follows immediately by applying $\sigma^{m}$ on both sides of the above equation.
\end{proof}

Let $r(\mathbf{a})=\sum_{i=1}^{2n}[a_i]_{+}$  for $\mathbf{a}=(a_1, a_2, \dots, a_{2n})\in \mathbb{Z}^{2n}$. Berenstein and Zelevinsky ~\cite{BZ2014} defined a partial order on $\mathbb{Z}^{2n}$ by setting 
$$\mathbf{a}'\prec \mathbf{a}\ \  \text{if and only if}\ \  r(\mathbf{a}') < r(\mathbf{a}).$$

\begin{lemma}\label{step-2}
For any indecomposable rigid object $M$ with $\Dim M=(m_1,m_2, \dots, m_{2n})$, we have that
$$X_M=q^{*} \prod \limits_{i=1}^{2n}X_{S_i}^{[m_i]_{+}}\prod \limits_{i=1}^{2n}X_{P_i[1]}^{[-m_i]_{+}}+\sum_{\mathbf{a}
\prec \Dim M} f_{\mathbf{a}}\prod \limits_{i=1}^{2n}X_{S_i}^{[a_i]_{+}}\prod \limits_{i=1}^{2n}X_{P_i[1]}^{[-a_i]_{+}},$$
where ${*} ={\frac{l}{2}}$ for some $l\in \mathbb{Z}$, $\mathbf{a}=(a_1, a_2, \dots, a_{2n})\in \mathbb{Z}^{2n}$ and $f_{\mathbf{a}} \in \mathbb{Z}[q^{\pm \frac{1}{2}}]$.
\end{lemma}

\begin{proof}
Due to the existence of short exact sequences
\[ \xymatrix{0 \ar[r]& P_{2n}  \ar[r] & E_1 \ar[r]& I_{1} \ar[r]&0}\] and
\[ \xymatrix{0 \ar[r]& P_{2n}   \ar[r] &  \tau I_1 \ar[r]& I_2\oplus I_{2n-1} \ar[r]&0},  \]
we obtain that $X_{I_1} X_{P_{2n}}=q^{\frac{1}{2}} X_{E_1}+ q^{-\frac{1}{2}} X_{I_2[-1]\oplus I_{2n-1}[-1]}$.
Note that $I_1=S_1, P_{2n}=S_{2n}$,  therefore the statement holds for $X_{E_1}.$

For any indecomposable regular rigid object $R_{E_{i},r}$ for $r\geq 1$, consider 
the following short exact sequences
\[ \xymatrix{0 \ar[r]& E_{i+1}  \ar[r] & R_{E_{i+1}, r+1} \ar[r]& R_{E_{i}, r} \ar[r]&0}\] and
\[ \xymatrix{0 \ar[r]& E_{i+1}   \ar[r] & R_{E_{i+1},r} \ar[r]& R_{E_{i},r-1}  \ar[r]&0},  \]
we obtain that 
\begin{align*}
X_{R_{E_{i}, r}} X_{E_{i+1}}= &q^{\frac{1}{2} \Lambda ((\mathbf{I}-R^{tr})\Dim R_{E_{i},r}, (\mathbf{I}-R^{tr})\Dim E_{i+1}) } X_{R_{E_{i+1}, r+1}}\\
&+ q^{\frac{1}{2} \Lambda ((\mathbf{I}-R^{tr})\Dim R_{E_{i},r}, (\mathbf{I}-R^{tr})\Dim E_{i+1}) -1} X_{R_{E_{i-1},r-1}}.
\end{align*}
Combining~\cite[Lemma 2.1]{BZ2014},  the statement holds for  $X_{R_{E_{i},r}}.$

By Lemma~\ref{step-1} and~\cite[Lemma 2.1]{BZ2014}, the statement holds for  all $X_M$ where $M$ is any indecomposable preprojective or preinjective representation. The proof is completed.
\end{proof}
Lemma \ref{step-2} implies that the quantum cluster algebra $\mathcal{A}_q(Q)$ is equal to its lower bound which is a particular case in \cite{BZ2005}.
\begin{proposition}
The set $\mathcal{B}$ is a $\mathbb{Z}[\displaystyle q^{\pm\frac{1}{2}}]$-basis of $\mathcal{A}_q(Q)$.
\end{proposition}
\begin{proof}
Let  $M$ be a rigid object in $\mathcal{C}_Q$  with $\Dim M=(m_1,m_2, \dots, m_{2n})$, according to Lemma \ref{step-2}, it is easy to deduce  that 
$$X_M=q^{\star} \prod \limits_{i=1}^{2n}X_{S_i}^{[m_i]_{+}}\prod \limits_{i=1}^{2n}X_{P_i[1]}^{[-m_i]_{+}}+\sum_{\mathbf{a}
\prec \Dim M} f_{\mathbf{a}}\prod \limits_{i=1}^{2n}X_{S_i}^{[a_i]_{+}}\prod \limits_{i=1}^{2n}X_{P_i[1]}^{[-a_i]_{+}},$$
where ${\star} ={\frac{n_1}{2}}$   for some $n_1\in \mathbb{Z}$, $\mathbf{a}=(a_1, a_2, \dots, a_{2n})\in \mathbb{Z}^{2n}$ and $f_{\mathbf{a}}\in \mathbb{Z}[q^{\pm \frac{1}{2}}]$.

Let $E$ be a regular rigid object in $\mathcal{C}_Q$  and  denote by $m\delta+\Dim E=(l_1, l_2, \dots, l_{2n})$ for $m\geq 1$. Note that the first term of the integer coefficient polynomials $F_{m}(x)$  is $x^{m}$.  Then by Lemma \ref{step-2} and Proposition \ref{delta-2}, we have
 $$F_{m}(X_\delta) X_E=q^{\#} \prod \limits_{i=1}^{2n}X_{S_i}^{[l_i]_{+}}\prod \limits_{i=1}^{2n}X_{P_i[1]}^{[-l_i]_{+}}+\sum_{\mathbf{b}
\prec (l_1, l_2, \dots, l_{2n})} f_{\mathbf{b}}\prod \limits_{i=1}^{2n}X_{S_i}^{[b_i]_{+}}\prod \limits_{i=1}^{2n}X_{P_i[1]}^{[-b_i]_{+}},$$
where   ${\#} ={\frac{n_2}{2}}$  for some $n_2\in \mathbb{Z}$, $\mathbf{b}=(b_1, b_2, \dots, b_{2n})\in \mathbb{Z}^{2n}$ and
$f_{\mathbf{b}}\in \mathbb{Z}[q^{\pm \frac{1}{2}}]$.

Note that $$\mathbb{Z}^{2n}=\{\Dim M, m\delta+\Dim E|M \text{ is rigid in } \mathcal{C}_Q, m \geq 1, E \text{ is regular rigid in } \mathcal{C}_Q\},$$
and the standard monomial basis $\{\prod \limits_{i=1}^{2n}X_{S_i}^{[a_i]_{+}}\prod \limits_{i=1}^{2n}X_{P_i[1]}^{[-a_i]_{+}}|\mathbf{a}=(a_1, a_2, \dots, a_{2n})\in \mathbb{Z}^{2n}\}$ indexed by $(\mathbb{Z}^{2n}, \prec)$ satisfies the property that for any $\mathbf{a}\in \mathbb{Z}^{2n}$, the lengths of chains in $\mathbb{Z}^{2n}$ with top  element $\mathbf{a}$ are bounded from above \cite{BZ2005}. Therefore it follows that $\mathcal{B}$ is a $\mathbb{Z}[\displaystyle q^{\pm\frac{1}{2}}]$-basis of $\mathcal{A}_q(Q)$.
\end{proof}

\begin{proposition}\label{positive}
Every element in the set $\mathcal{B}$ is positive.
\end{proposition}
\begin{proof}
We only need to prove that $F_{m}(X_\delta)$ is positive. For any cluster $\{X_{T_1}, X_{T_2}, \dots, X_{T_{2n}}\}$, 
there exists some  $1 \leq i \leq 2n$, such that
by  applying the automorphism $\sigma$  to the equations in Theorem~\ref{mainresult1}, we have 
$F_{m}(X_\delta)X_{T_i}=q^{\frac{l_1}{2}}X_{L_i}+q^{\frac{l_1}{2}}X_{L_i'}$  for some $l_1,l_2\in \mathbb{Z}$,
 and $X_{L_i}$ and $X_{L_i'}$ are certain cluster variables. Due to the positivity of quantum cluster variables~\cite{Davison,KQ},  we know that both  $X_{L_i}$ and $X_{L_i'}$ belong to 
$\mathbb{Z}_{\geq 0}[q^{\pm \frac{1}{2}}][X_{T_1}, X_{T_2}, \dots, X_{T_{2n}}]$ which implies that $F_{m}(X_\delta) \in \mathbb{Z}_{\geq 0}[q^{\pm \frac{1}{2}}][X_{T_1}, X_{T_2}, \dots, X_{T_{2n}}]$.
\end{proof}

\begin{lemma}\label{object} We have the following:
\begin{enumerate}
\item For any rigid object $M$, there exists some cluster tilting object $T=T_{1} \oplus T_{2} \cdots \oplus T_{2n}$ with  $M$ as a direct summand of $lT$ for some $l\in \mathbb{Z}_{\geq 1}$ such that $X_M$ does not appear 
in the $\{X_{T_{1}}, X_{T_{2}}, \dots, X_{T_{2n}}\}$-expansion of any other basis element of $\mathcal{B}$; 
\item Let $E$ be a regular rigid object and $m\in \mathbb{Z}_{\geq 1}$, then there exist infinitely many cluster tilting 
objects $T^{(r)}=T^{(r)}_{1} \oplus T^{(r)}_{2} \cdots \oplus T^{(r)}_{2n}$ for $r\in \mathbb{Z}$ such that
we can choose a laurent monomial $Y_{{m,E,r}}$ in the $\{X_{T^{(r)}_{1} }, X_{T^{(r)}_{2} }, \dots, X_{T^{(r)}_{2n} }\}$-expansion of  $F_{m}(X_\delta)X_E$ with coefficients  $q^{\frac{l}{2}}$ for certain $l\in \mathbb{Z}$. 
Moreover, for any other basis element  $b$, if $r$ is sufficiently large, then $Y_{{m,E,r}}$ does not appear in the $\{X_{T^{(r)}_{1} }, X_{T^{(r)}_{2} }, \dots, X_{T^{(r)}_{2n} }\}$-expansion of $b$.
\end{enumerate}
\end{lemma}
\begin{proof}
Note that the statements (1) and (2) hold for classical cluster algebras of type $\widetilde{A}_{p,q}$~\cite{DTh}.
For any cluster $\{X_{T_1}, X_{T_2}, \dots, X_{T_{2n}}\}$ and any basis element $b$, by Proposition~\ref{positive}, we have that
$$b=\sum_{\mathbf{a}} \lambda_{\mathbf{a}}(q^{\pm \frac{1}{2}})X_{T_1}^{a_1}  X_{T_2}^{a_2} \cdots X_{T_{2n}}^{a_{2n}}$$ where $\mathbf{a}=(a_1, a_2, \dots, a_{2n})\in \mathbb{Z}^{2n}$ and $\lambda_{\mathbf{a}} \in \mathbb{Z}_{\geq 0}[q^{\pm \frac{1}{2}}]$. Setting $q=1$, we have $\lambda_{\mathbf{a}}(1) \in \mathbb{Z}_{\geq 1}$ if $\lambda_{\mathbf{a}} (q^{\pm \frac{1}{2}}) \neq 0$. This implies that every term $X_{T_1}^{a_1}  X_{T_2}^{a_2} \cdots X_{T_{2n}}^{a_{2n}}|_{q=1}$ appears in the expansion of
the basis element $b|_{q=1}$ of classical cluster algebra. 

We only prove statement (1). Similar arguments apply to (2). We choose the same cluster tilting object $T$ as in the corresponding classical cluster algebra~\cite{DTh}. If $X_M$ appears in the $\{X_{T_{1}}, X_{T_{2}}, \dots, X_{T_{2n}}\}$-expansion of some other basis element $b'$ of 
$\mathcal{B}$, then by the above 
discussions, $X_M|_{q=1}$ appears
in the $\{X_{T_{1}}, X_{T_{2}}, \dots, X_{T_{2n}}\}|_{q=1}$-expansion of the atomic basis element $b'|_{q=1}$ in the classical case, which is a contradiction. Thus the statement (1) follows immediately.
\end{proof}

\begin{theorem}\label{mainresult2}
The set $\mathcal{B}$
is a bar-invariant atomic basis of the quantum cluster algebra $\mathcal{A}_q(Q)$.
\end{theorem}
\begin{proof}
In order to prove $\mathcal{B}$ is bar-invariant, we only need to prove that the elements in  $\{F_{m}(X_\delta) X_E | m \geq 1, E \text{ is a regular rigid object in } \mathcal{C}_Q \}$ are bar-invariant which immediately follows from Proposition \ref{bar-inv} and the definition of the $m$-th Chebyshev polynomials of the first kind $F_m(x), m\geq 0$. 

Let $Y$ be a positive element in $\mathcal{A}_q(Q)$, then
$$Y=\sum_{M \text{ rigid}} \lambda_M X_M+ \sum_{\tiny\begin{array}{c} m  \in \mathbb{Z}_{\geq 1}, \\ E \text{ regular rigid} \end{array}} \lambda_{m, E} F_{m}(X_\delta) X_E$$
where $\lambda_M,  \lambda_{m, E} \in \mathbb{Z}[\displaystyle q^{\pm \frac{1}{2}}]$. 

By Lemma~\ref{object}(1), 
we can find some cluster tilting object $T=T_{1} \oplus T_{2} \cdots \oplus T_{2n}$ with  $M$ as a direct summand of $lT$ for some $l\in \mathbb{Z}_{\geq 1}$ such that $X_M$ does not appear 
in the $\{X_{T_{1}}, X_{T_{2}}, \dots, X_{T_{2n}}\}$-expansion of any other basis element of $\mathcal{B}$. Thus $\lambda_M$ coincides with the coefficient of $X_M$ in the $\{X_{T_{1}}, X_{T_{2}}, \dots, X_{T_{2n}}\}$-expansion of $Y$. Note that $Y$ is assumed to be positive, we have that  $\lambda_M \in \mathbb{Z}_{\geq 0}[\displaystyle q^{\pm \frac{1}{2}}]$

By Lemma~\ref{object}(2),  by choosing sufficiently large $r$, we can find a laurent monomial $Y_{{m,E,r}}$ in the $\{X_{T^{(r)}_{1} }, X_{T^{(r)}_{2} }, \dots, X_{T^{(r)}_{2n} }\}$-expansion of  $F_{m}(X_\delta)X_E$ with coefficients  $q^{\frac{l'}{2}}$ for certain $l' \in \mathbb{Z}$, but not appear in the $\{X_{T^{(r)}_{1} }, X_{T^{(r)}_{2} }, \dots, X_{T^{(r)}_{2n} }\}$-expansion of any other basis element in these sum terms. Thus $q^{\frac{l'}{2}}\lambda_{m, E} $ coincides with the coefficient of $Y_{{m,E,r}}$ in the $\{X_{T^{(r)}_{1} }, X_{T^{(r)}_{2} }, \dots, X_{T^{(r)}_{2n} }\}$-expansion of $Y$. Note that $Y$ is assumed to be positive, we have that  $\lambda_{m,E} \in \mathbb{Z}_{\geq 0}[\displaystyle q^{\pm \frac{1}{2}}]$.
The proof is completed.
\end{proof}

\begin{remark}
When $n=1$, the quiver $\widetilde{A}_{2n-1,1}$ is the Kronecker quiver, the quantum cluster algebra of the Kronecker quiver $\widetilde{A}_{1,1}$ has been studied in details in~\cite{dx}.
\end{remark}

\subsection{Another bar-invariant $\mathbb{Z}[q^{\pm\frac{1}{2}}]$-basis} 
The $m$-th Chebyshev polynomials of the second kind $S_{m}(x)$ is related to dual semicanonical bases of cluster algebras, and defined by
  $$
  S_0(x)=1,S_1(x)=x, S_2(x)=x^2-1, S_{m+1}(x)=S_{m}(x)x-S_{m-1}(x)~\text{for}~m\geq2.
 $$
We set
$$
 \begin{aligned}
&\mathcal{S}= \{X_M | M \text{ is rigid in } \mathcal{C}_Q \} \cup \{S_{m}(X_\delta) X_E | m \geq 1, E \text{ is a regular rigid object in } \mathcal{C}_Q\}.\\
  \end{aligned}
$$
\begin{theorem}\label{mainresult-1}
The set $\mathcal{S}$
is a bar-invariant $\mathbb{Z}[q^{\pm\frac{1}{2}}]$-basis of the quantum cluster algebra $\mathcal{A}_q(Q)$. Moreover, every basis element  is positive.
\end{theorem}
\begin{proof}
Since the $m$-th Chebyshev polynomials of  the first kind and  the second kind are related by $$F_m(X_\delta)=S_{m}(X_\delta)-S_{m-2}(X_\delta),$$
it follows  that $\mathcal{S}$ is also a $\mathbb{Z}[\displaystyle q^{\pm\frac{1}{2}}]$-basis of $\mathcal{A}_q(Q)$. 

By Proposition \ref{bar-inv} and the definition of the $m$-th Chebyshev polynomials of the second kind,
the elements in  $\{S_{m}(X_\delta) X_E | m \geq 1, E \text{ is a regular rigid object in } \mathcal{C}_Q \}$ are bar-invariant.  

By Proposition \ref{positive}, $F_m(X_\delta)X_E$ is positive. Then by combining with $F_m(X_\delta)=S_{m}(X_\delta)-S_{m-2}(X_\delta)$, we can obtain that the elements in  $S_m(X_\delta)X_E$ are  positive. Thus  the proof is completed. 
\end{proof}

The following result gives an representation-theoretic interpretation of the elements in  $\{S_{m}(X_\delta) X_E | m \geq 1, E \text{ is a regular rigid object in } \mathcal{C}_Q \}$.

\begin{proposition} For any $m\geq 1$, we have 
$$S_m(X_\delta)X_E=X_{R_{E(\lambda), m}\oplus E}.$$
\end{proposition}
\begin{proof}
For any $\lambda \in \mathbb{F}_2$, due to short exact sequences
\[ \xymatrix{0 \ar[r]& E(\lambda) \ar[r] & R_{E(\lambda),m+1} \ar[r]& R_{E(\lambda),m}\ar[r]&0}\] and
\[ \xymatrix{0 \ar[r]& E(\lambda) \ar[r] & R_{E(\lambda),m}\ar[r]& R_{E(\lambda),m-1} \ar[r]&0}, \] we have that
\begin{align*}
& X_{R_{E(\lambda),m}} X_{E(\lambda)} \\
= &\displaystyle q^{\frac{1}{2}  \Lambda ((\mathbf{I}-R^{tr})n\delta, (\mathbf{I}-R^{tr})\delta)  } X_{R_{E(\lambda),m+1}}        +
\displaystyle q^{\frac{1}{2}  \Lambda ((\mathbf{I}-R^{tr})n\delta, (\mathbf{I}-R^{tr})\delta) +\frac{1}{2} \langle n\delta, \delta \rangle } X_{R_{E(\lambda),m-1}} \\
=& X_{R_{E(\lambda),m+1}} +X_{R_{E(\lambda),m-1}}.
\displaystyle 
\end{align*}
By comparing with the definition of $S_m(X_\delta)$, we can deduce that $S_m(X_\delta)=X_{R_{E(\lambda),m}}$. According to  Lemma \ref{lamda}, we can then get
$$S_m(X_\delta) X_E=X_{R_{E(\lambda),m}} X_E=\displaystyle q^{\frac{1}{2}  \Lambda ((\mathbf{I}-R^{tr})m\delta, (\mathbf{I}-R^{tr})\Dim E)  } X_{R_{E(\lambda),m}\oplus E}=X_{R_{E(\lambda),m}\oplus E}.$$ The proof is completed.
\end{proof}
\begin{remark}
For Kronecker quiver $Q=\widetilde{A}_{1,1}$, the set $\mathcal{S}$ is shown to be the triangular basis of the quantum cluster algebra $\mathcal{A}_q(Q)$ \cite{BZ2014,fanqin2}. We conjecture that it holds for type $\widetilde{A}_{2n-1,1}, n\geq 2$.
\end{remark}

\appendix
\section{Automorphism of quantum cluster algebra induced by shift functor}
In this appendix, we consider the equal-valued acyclic quiver $(Q,\bf d)$ of full rank with $n$ vertices and $\bf d$$=(d,d,...,d)$ which is obtained from a compatible pair $(\Lambda, \tilde{B})$ with $\tilde{B}=B$. In the following, let $\mathbf{I}:=\mathbf{I}_n$ be the identity matrix of size $n$.
Let $I$ and $P$ be the injective and projective representation of $(Q,\bf d)$ such that $\operatorname{soc} I \simeq P/\rad P$. Note that $(\mathbf{I}-R^{tr})\Dim I= \Dim \soc I=\Dim P/{\operatorname{rad} P}=(\mathbf{I}-R) \Dim P$. For any representation $M$,  we have that $(\mathbf{I}-R^{tr})\Dim \tau M=-(\mathbf{I}-R) \Dim M$. All notations are same as in previous sections.

\begin{lemma}\label{fornext} Let $M$, $N$, $\tau M$ and $\tau N$ be representations of $(Q,\bf d)$, then we have that 
\begin{align*}
(1)\,\,\, &  \Lambda ( (\mathbf{I}-R^{tr})\Dim \tau M,  ( \mathbf{I}-R^{tr})\Dim \tau N)=\Lambda ( (\mathbf{I}-R^{tr})\Dim M,  ( \mathbf{I}-R^{tr})\Dim N)\\
(2)\,\,\, &  \Lambda ( (\mathbf{I}-R^{tr})\Dim I_j,  ( \mathbf{I}-R^{tr})\Dim I_i)  \\
     &   =\Lambda ( (\mathbf{I}-R^{tr})\Dim I_j,  ( \mathbf{I}-R^{tr})\Dim P_i)-d({\Dim \soc} I_j)^{tr}\Dim P_i; \\
(3) \,\,\, &  \Lambda ( (\mathbf{I}-R^{tr})\Dim I_i,  ( \mathbf{I}-R^{tr})\Dim \tau M)\\
&=\Lambda ( (\mathbf{I}-R^{tr})\Dim M,  ( \mathbf{I}-R^{tr})\Dim I_i)+d(\Dim \soc I_i)^{tr}\Dim M;    \\
(4) \,\,\, & \Lambda ( (\mathbf{I}-R^{tr})\Dim I_i,  ( \mathbf{I}-R^{tr})\Dim \tau M)=\Lambda ( (\mathbf{I}-R^{tr})\Dim M,  ( \mathbf{I}-R^{tr})\Dim P_i).
\end{align*}


\end{lemma}
\begin{proof}

We can calculate directly as follows:
\begin{enumerate} 
\item \begin{align*}
&\Lambda ( (\mathbf{I}-R^{tr})\Dim \tau M,  ( \mathbf{I}-R^{tr})\Dim \tau N)=\Lambda ( (\mathbf{I}-R)\Dim M,  ( \mathbf{I}-R)\Dim N) \\
= &\Lambda ( (\mathbf{I}-R^{tr})\Dim M+ B \Dim M,  ( \mathbf{I}-R^{tr})\Dim N+B \Dim N) \\
=& \Lambda ( (\mathbf{I}-R^{tr})\Dim M,  ( \mathbf{I}-R^{tr})\Dim N) + \Lambda ( (\mathbf{I}-R^{tr})\Dim M,  B\Dim N)\\
& +\Lambda (B \Dim M,  ( \mathbf{I}-R^{tr})\Dim N)+   \Lambda ( B \Dim M,  B \Dim N)   \\
=& \Lambda ( (\mathbf{I}-R^{tr})\Dim M,  ( \mathbf{I}-R^{tr})\Dim N) +     d(\Dim M)^{tr}  (R-\mathbf{I})\Dim N\\
& +  d({\Dim M})^{tr} ( \mathbf{I}-R^{tr}) \Dim N+   d(\Dim M)^{tr}  B \Dim N    \\
=& \Lambda ( (\mathbf{I}-R^{tr})\Dim M,  ( \mathbf{I}-R^{tr})\Dim N);
\end{align*}
\item \begin{align*}
&\Lambda ( (\mathbf{I}-R^{tr})\Dim I_j,  ( \mathbf{I}-R^{tr})\Dim I_i)=\Lambda ( (\mathbf{I}-R^{tr})\Dim I_j,  ( \mathbf{I}-R)\Dim P_i) \\
= &\Lambda ( (\mathbf{I}-R^{tr})\Dim I_j,  ( \mathbf{I}-R^{tr})\Dim P_i) +  \Lambda ( (\mathbf{I}-R^{tr})\Dim I_j,  B\Dim P_i)   \\
=& \Lambda ( (\mathbf{I}-R^{tr})\Dim I_j,  ( \mathbf{I}-R^{tr})\Dim P_i) + ({\Dim \soc} I_j)^{tr}\Lambda  B\Dim P_i \\
=& \Lambda ( (\mathbf{I}-R^{tr})\Dim I_j,  ( \mathbf{I}-R^{tr})\Dim P_i) -d ({\Dim \soc} I_j)^{tr}\Dim P_i;
\end{align*} 
\item \begin{align*}
&\Lambda ( (\mathbf{I}-R^{tr})\Dim I_i,  ( \mathbf{I}-R^{tr})\Dim \tau M)=\Lambda ( (\mathbf{I}-R^{tr})\Dim I_i,  -( \mathbf{I}-R)\Dim M) \\
= &\Lambda ( (\mathbf{I}-R^{tr})\Dim I_i, - ( \mathbf{I}-R^{tr})\Dim M) +  \Lambda ( (\mathbf{I}-R^{tr})\Dim I_i,  -B\Dim M)   \\
=& \Lambda ( (\mathbf{I}-R^{tr})\Dim M,  ( \mathbf{I}-R^{tr})\Dim I_i)+d({\Dim \soc} I_i)^{tr}\Dim M;
\end{align*} 
\item \begin{align*}
&\Lambda ( (\mathbf{I}-R^{tr})\Dim I_i,  ( \mathbf{I}-R^{tr})\Dim \tau M)=\Lambda ( (\mathbf{I}-R)\Dim P_i,  -( \mathbf{I}-R)\Dim M) \\
= &\Lambda ( (\mathbf{I}-R^{tr})\Dim P_i, - ( \mathbf{I}-R^{tr})\Dim M) +  \Lambda ( (\mathbf{I}-R^{tr})\Dim P_i,  -B\Dim M)   \\
&+ \Lambda ( B\Dim P_i, -(\mathbf{I}-R^{tr})\Dim M)+  \Lambda ( B\Dim P_i, -B\Dim M)  \\
=& \Lambda ( (\mathbf{I}-R^{tr})\Dim M,  (\mathbf{I}-R^{tr})\Dim P_i).
\end{align*} 
\end{enumerate}\end{proof}

\begin{theorem}\label{auto} 
The mapping $\sigma: \mathcal{A}_q (Q) \longrightarrow \mathcal{A}_q (Q)$ which sends $X_M$ to $X_{M[1]}$ for any  rigid object $M$ in $\mathcal{C}_Q$ is  an automorphism.
\end{theorem}

\begin{proof}
We  need to prove that $\sigma$ preserves the following  two classes of relations of $\mathcal{A}_q (Q)$. 
\begin{enumerate} 
\item  Preserving the quasi-commuting relations. 

Let $M\oplus N$ be a rigid object in $\mathcal{C}_Q$. 
We only consider the case: $M$, $N$, $M[1]$  and $N[1]$ are all representations of $(Q, \bf d)$. The proof are similar for other cases.  Then the quasi-commuting relations are
$$
X_MX_N =\displaystyle q^{\frac{1}{2}  \Lambda ( (\mathbf{I}-R^{tr})\Dim  M,  ( \mathbf{I}-R^{tr})\Dim N)}X_{M\oplus N};$$
$$X_M X_N =\displaystyle q^{\Lambda ( (\mathbf{I}-R^{tr})\Dim  M,  ( \mathbf{I}-R^{tr})\Dim N)}X_NX_M.$$
By using Lemma~\ref{fornext} (1), we obtain that
\begin{align*}
& X_{M[1]} X_{N[1]} =\displaystyle q^{\frac{1}{2}  \Lambda ( (\mathbf{I}-R^{tr})\Dim  M[1],  ( \mathbf{I}-R^{tr})\Dim N[1])}X_{M[1]\oplus N[1]}\\
= & \displaystyle q^{\frac{1}{2}  \Lambda ( (\mathbf{I}-R^{tr})\Dim  M,  ( \mathbf{I}-R^{tr})\Dim N)}X_{M[1]\oplus N[1]};
\end{align*} 
\begin{align*}
& X_{M[1]} X_{N[1]} =\displaystyle q^{\Lambda ( (\mathbf{I}-R^{tr})\Dim  M[1],  ( \mathbf{I}-R^{tr})\Dim N[1])}X_{M[1]}X_{N[1]}\\
= & \displaystyle q^{\frac{1}{2}  \Lambda ( (\mathbf{I}-R^{tr})\Dim  M,  ( \mathbf{I}-R^{tr})\Dim N)}X_{M[1]}X_{N[1]},
\end{align*} 
which imply that   $\sigma$ preserves the  relations in this case.

\item Preserving the exchange relations.  

Firstly let $M$ and $N$ be indecomposable rigid representation of $(Q, \bf d)$ with \\ $\dim_{\End(M)} \Ext^1 (M,N)=1.$  
Then the exchange relation is 
\begin{align*}
&X_M X_N  \\
= & \displaystyle q^{\frac{1}{2}  \Lambda ( (\mathbf{I}-R^{tr})\Dim  M,  ( \mathbf{I}-R^{tr})\Dim N)}X_E \\
& +q^{\frac{1}{2}  \Lambda ( (\mathbf{I}-R^{tr})\Dim  M,  ( \mathbf{I}-R^{tr})\Dim N)-\frac{1}{2} d}X_{D\oplus A\oplus I[-1] },
\end{align*} 
 where the corresponding exact sequences are
\[ \xymatrix{0 \ar[r]& N  \ar[r] & E \ar[r]& M \ar[r]&0}\] and
\[ \xymatrix{0 \ar[r]& D \ar[r] &  N \ar[r]& \tau M \ar[r]&   \tau A\oplus I\ar[r]   & 0}. \]
If  $M[1]$  and $N[1]$ are also indecomposable rigid representation of $(Q,  \bf d)$,  then by using Lemma~\ref{fornext} (1), we obtain that
\begin{align*}
&X_{M[1]} X_{N[1]}  \\
= & \displaystyle q^{\frac{1}{2}  \Lambda ( (\mathbf{I}-R^{tr})\Dim  M[1],  ( \mathbf{I}-R^{tr})\Dim N[1])}X_{E[1]} \\
&+q^{\frac{1}{2}  \Lambda ( (\mathbf{I}-R^{tr})\Dim  M[1],  ( \mathbf{I}-R^{tr})\Dim N[1])-\frac{1}{2} d }X_{D[1]\oplus A[1]\oplus I }\\
= & \displaystyle q^{\frac{1}{2}  \Lambda ( (\mathbf{I}-R^{tr})\Dim  M,  ( \mathbf{I}-R^{tr})\Dim N)}X_{E[1]} \\
& +q^{\frac{1}{2}  \Lambda ( (\mathbf{I}-R^{tr})\Dim  M,  ( \mathbf{I}-R^{tr})\Dim N)-\frac{1}{2} d }X_{D[1]\oplus A[1]\oplus I },
\end{align*} 
which implies that   $\sigma$ preserves the  relations in this case. The proof of other cases is similar.

Secondly let $M$ be an indecomposable rigid representation of  $(Q, \bf d)$ and $I$ be an indecomposable injective representation of $(Q, \bf d)$ with 
$\dim_{\End(I)} \Hom (M,I)=\dim_{\End(P)} \Hom (P,W)=1$ where
$P=\nu^{-1} (I)$ with $\nu^{-1}$ being the inverse Nakayama functor. Then the exchange relation is  
\begin{align*}
&X_M X_{I[-1]} \\
= & \displaystyle q^{-\frac{1}{2}  \Lambda ( (\mathbf{I}-R^{tr})\Dim  M,  ( \mathbf{I}-R^{tr})\Dim I)}X_{G\oplus I'[-1]} \\ &+
q^{-\frac{1}{2}  \Lambda ( (\mathbf{I}-R^{tr})\Dim  M,  ( \mathbf{I}-R^{tr})\Dim I)-\frac{1}{2} d}X_{F\oplus P''[1] },
\end{align*} 
where the corresponding exact sequences are
\[ \xymatrix{0 \ar[r]& G \ar[r] &  M \ar[r]& I \ar[r]&  I'\ar[r]   & 0}\] and
\[ \xymatrix{0 \ar[r]& P'' \ar[r] &  P \ar[r]& M \ar[r]&   F\ar[r]   & 0}. \]
In the following, we only consider the case $M$ is not projective, and the proof  of the other case is similar.  Thus we obtain that
\begin{align*}
&X_{I} X_{M[1]} \\
= & \displaystyle q^{\frac{1}{2}  \Lambda ( (\mathbf{I}-R^{tr})\Dim  I,  ( \mathbf{I}-R^{tr})\Dim \tau M)}X_{F[1]\oplus I''} \\ &+
q^{\frac{1}{2}  \Lambda ( (\mathbf{I}-R^{tr})\Dim  I,  ( \mathbf{I}-R^{tr})\Dim \tau M)-\frac{1}{2} d}X_{G[1]\oplus I' },
\end{align*} 
where $P''=\nu^{-1} (I'')$.
Then,  by using bar-involution and Lemma~\ref{fornext} (3),  we have
\begin{align*}
&X_{M[1]} X_{I} \\
= & \displaystyle q^{-\frac{1}{2}  \Lambda ( (\mathbf{I}-R^{tr})\Dim  I,  ( \mathbf{I}-R^{tr})\Dim \tau M)}X_{F[1]\oplus I''} \\ &+
q^{-\frac{1}{2}  \Lambda ( (\mathbf{I}-R^{tr})\Dim  I,  ( \mathbf{I}-R^{tr})\Dim \tau M)+\frac{1}{2} d}X_{G[1]\oplus I' }\\
= & \displaystyle q^{-\frac{1}{2}  \Lambda ( (\mathbf{I}-R^{tr})\Dim  M,  ( \mathbf{I}-R^{tr})\Dim I)}X_{G[1]\oplus I'} \\ &+
q^{-\frac{1}{2}  \Lambda ( (\mathbf{I}-R^{tr})\Dim  M,  ( \mathbf{I}-R^{tr})\Dim I)-\frac{1}{2} d}X_{F[1]\oplus I''},\end{align*} 
which imply that   $\sigma$ preserves the  relations in this case. 

\end{enumerate}

Hence we can verify $\sigma$ is an automorphism.
\end{proof}


\end{document}